\documentclass[12pt]{article}
\usepackage[utf8]{inputenc}
\usepackage[T2A]{fontenc}

\usepackage[english]{babel}
\usepackage{fullpage}

\usepackage{cmap}					
\usepackage{mathtext}

\usepackage{amsmath,amsfonts,amssymb,amsthm,mathtools} 
\usepackage{icomma}

\usepackage{mathrsfs}

\usepackage{bbm}
\usepackage{xcolor}
\usepackage{hyperref}
\usepackage{bm}
\usepackage{cite}

\usepackage{epstopdf}
\usepackage{subfigure}

\theoremstyle{plain}
\newtheorem{theorem}{Theorem}[section]
\newtheorem{corollary}[theorem]{Corollary}
\newtheorem{lemma}[theorem]{Lemma}

\theoremstyle{definition}
\newtheorem{definition}{Definition}

\theoremstyle{remark}

\newcommand{\la}{\langle}
\newcommand{\ra}{\rangle}

\mathtoolsset{showonlyrefs=true}

\usepackage{graphicx}  
\graphicspath{{images/}{images2/}}  
\usepackage{wrapfig} 

\usepackage{array,tabularx,tabulary,booktabs} 
\usepackage{longtable}  
\usepackage{multirow} 

\usepackage{etoolbox} 

\usepackage{comment}
\usepackage{lastpage} 

\usepackage{soul}

\usepackage{multicol}

\newcommand{\argmin}{\operatornamewithlimits{argmin}}

\usepackage{algorithm}
\usepackage[noend]{algpseudocode}

\begin{document}

\title{Universal Intermediate Gradient Method for Convex Problems with Inexact Oracle}
 
\author{
Dmitry Kamzolov \textsuperscript{a}
Pavel Dvurechensky\textsuperscript{b} and Alexander Gasnikov\textsuperscript{a,c}.\\
\textsuperscript{a}Moscow Institute of Physics and Technology, Moscow, Russia;\\
\textsuperscript{b} Weierstrass Institute for Applied Analysis and Stochastics, Berlin, Germany;\\ 
\textsuperscript{c}Institute for Information Transmission Problems RAS, Moscow, Russia.
}

\maketitle

\begin{abstract}
In this paper, we propose new first-order methods for minimization of a convex function on a simple convex set. We assume that the objective function is a composite function given as a sum of a simple convex function and a convex function with inexact H\"older-continuous subgradient. We propose Universal Intermediate Gradient Method. Our method enjoys both the universality and intermediateness properties. Following the ideas of Y. Nesterov (Math.Program. 152: 381-404, 2015) on Universal Gradient Methods, our method does not require any information about the H\"older parameter and constant and adjusts itself automatically to the local level of smoothness. On the other hand, in the spirit of the Intermediate Gradient Method proposed by O. Devolder, F.Glineur and Y. Nesterov (CORE Discussion Paper 2013/17, 2013), our method is intermediate in the sense that it interpolates between Universal Gradient Method and Universal Fast Gradient Method. This allows to balance the rate of convergence of the method and rate of the oracle error accumulation. Under additional assumption of strong convexity of the objective, we show how the restart technique can be used to obtain an algorithm with faster rate of convergence.
\end{abstract}

\section{Introduction}

In this paper, we consider first-order methods for minimization of a convex function over a simple convex set. The renaissance of such methods started more than ten years ago and was mostly motivated by large-scale problems in data analysis, imaging and machine learning. Simple black-box oriented methods like Mirror Descent \cite{nemirovsky1983problem} or Fast Gradient Method \cite{nesterov1983method}, which were known in the 1980s, got a new life. 

For a long time algorithms and their analysis were, mostly, separate for two main classes of problems. The first class, with optimal method being Mirror Descent, is the class of non-smooth convex functions with bounded subgradients. The second is the class of smooth convex functions with Lipschitz-continuous gradient, and the optimal method for this class is Fast Gradient Method. An intermediate class of problems with H\"older-continuous subgradient was also considered and optimal methods for this class were proposed in \cite{nemirovskii1985optimal}. However, these methods require to know the H\"older constant. In 2013, Nesterov proposed a Universal Fast Gradient Method \cite{nesterov2015universal} which is free of this drawback and is uniformly optimal for the class of convex problems with H\"older-continuous subgradient in terms of black-box information theoretic lower bounds \cite{nemirovsky1983problem}. In 2012, Lan proposed a Fast gradient method with one prox-mapping for stochastic optimization problems \cite{lan2012optimal}. In 2016, Gasnikov and Nesterov proposed a Universal Triangle Method \cite{gasnikov2016universal}, which possesses all the properties of Universal Fast Gradient Method, but uses only one proximal mapping instead of two, as opposed to the previous version. We also mention the work \cite{lan2015generalized}, where the authors introduce a method which is uniformly optimal for convex and non-convex problems with H\"older-continuous subgradient, and the work \cite{yurtsever2015universal}, in which a universal primal-dual method is proposed to solve linearly constrained convex problems.

Another line of research \cite{d2008smooth,devolder2013intermediate,devolder2014first, dvurechensky2016stochastic,gasnikov2015gradient} studies first-order methods with inexact oracle. The considered inexactness can be of deterministic or stochastic nature, it can be connected to inexact calculation of the subgradient or to inexact solution of some auxiliary problem. As it was shown in \cite{devolder2014first}, gradient descent has slower rate of convergence, but does not accumulate the error of the oracle. On the opposite, Fast Gradient Method has faster convergence rate, but accumulates the error linearly with the iteration counter. Later, in \cite{devolder2013intermediate} an Intermediate Gradient Method was proposed. The main feature of this method is that, depending on the choice of a hyperparameter, it interpolates between Gradient Method and Fast Gradient Method to exploit the trade-off between the rate of convergence and the rate of error accumulation.

In this paper, we join the above two lines of research and present Universal Intermediate Gradient Method (UIGM) for problems with deterministic inexact oracle. Our method enjoys both the universality with respect to smoothness of the problem and interpolates between Universal Gradient Method and Universal Fast Gradient Method, thus, allowing to balance the rate of convergence of the method and rate of the error accumulation. We consider a composite convex optimization problem on a simple set with convex objective, which has inexact H\"older-continuous subgradient, propose a method to solve it, and prove the theorem on its convergence rate. The obtained rate of convergence is uniformly optimal for the considered class of problems. This method can be used in different applications such as transport modeling \cite{baimurzina2017universal,gasnikov2015universal}, inverse problems \cite{gasnikov2017convex} and others. 

We also consider the same problem under additional assumption of strong convexity of the objective function and show how the restart technique \cite{nemirovsky1983problem,nemirovskii1985optimal,nesterov1983method,juditsky2014deterministic,kamzolov2016universal,roulet2017sharpness,fercoq2017adaptive} can be applied to obtain a faster convergence rate of UIGM.
The obtained rate of convergence is again optimal for the class of strongly convex functions with H\"older-continuous subgradient.

The rest of the paper is organized as follows. 
In Sect. 2, we state the problem. After that, in Sect. 3, we present Universal Intermediate Gradient Method and prove a convergence rate theorem with general choice of controlling sequence of coefficients. In Sect. 4, we analyze particular choice of controlling sequence of coefficients and prove a convergence rate theorem under this choice of coefficients. In Sect. 5, we present UIGM for strongly convex functions and prove convergence rate theorem under this additional assumption. In Sect. 6, we introduce another choice of coefficients that don't need any additional information. In Sect. 7, we present numerical experiments for our method.

\section{Problem Statement and Preliminaries}
In what follows, we work in a finite-dimensional linear vector space $E$. Its dual space, the space of all linear functions on $E$, is denoted by $E^{\ast}$. Relative interior of $Q$ is denoted as rint $Q$. For $x\in E$ and $s\in E^{\ast}$, we denote by $\left\langle s,x \right\rangle$ the value of a linear function $s$ at $x$. For the (primal) space $E$, we introduce a norm $\|\,\cdot\,\|_E$. Then the dual norm is defined in the standard way:
$$
\|s\|_{E,\ast}=\max\limits_{x\in E} \left\lbrace \left\langle s,x \right\rangle: \, \|x\|_E\leq 1 \right\rbrace.
$$
Finally, for a convex function $f: \textbf{dom } f \rightarrow R$ with $\textbf{dom } f \subseteq E$ we denote by $\nabla f(x)\in E^{\ast}$ one of its subgradients.

We consider the following convex composite optimization problem \cite{nesterov2013gradient}:
\begin{equation}
\min\limits_{x\in Q} \left[ F(x) \overset{\text{def}}{=}f(x)+h(x) \right],
\label{eq_pr}
\end{equation}
where $Q$ is a simple closed convex set, $h(x)$ is a simple closed convex function and $f(x)$ is a convex function on $Q$ with inexact first-order oracle, defined below. We assume that problem \eqref{eq_pr} is solvable with optimal solution $x^{\ast}$.

\begin{definition} 
We say that a convex function $f(x)$ is equipped with a \textit{first-order $\left(\delta,L  \right)$--oracle} on a convex set $Q$ if for any point $x \in Q$, $(\delta,L)$-oracle returns a pair $(f_{\delta}(x),\, g_{\delta}(x))\in R \times E^{\ast}$ such that
\begin{equation}
 0\leq f(y) - f_{\delta}(x) - \left\langle g_{\delta}(x), y-x\right\rangle \leq \frac{L}{2}\|y-x\|_E^2+\delta, \quad \forall y \in Q.
 \label{eq_or}
\end{equation}
\label{def_inex}
\end{definition}
In this definition, $\delta$ represents the error of the oracle  \cite{devolder2014first}. The oracle is exact with $\delta=0$. Also we can take $\delta=\delta_c+\delta_u$, where $\delta_c$ represents the error, which we can control and make as small as we would like to. On the opposite, $\delta_u$ represents the error, which we can not control \cite{dvurechensky2016stochastic}. 
Note that, by Definition \ref{def_inex}, 
\begin{equation}
\label{eq_fd}
0\leq f(x) - f_{\delta}(x)  \leq \delta, \quad \forall x \in Q.
\end{equation}

To motivate Definition \ref{def_inex}, we consider the following example. Let $f$ be a convex function with H\"older-continuous subgradient. Namely, there exists $\nu \in [0,1]$, and $M_{\nu}<+\infty$, such that
\begin{equation*}
\| \nabla f(x) - \nabla f(y) \|_{E^\ast} \leq M_{\nu}\|x-y\|_E^{\nu}, \quad \forall x,y \in Q.
\end{equation*}
In \cite{devolder2014first}, it was proved that, for such function for any $\delta_c>0$, if
\begin{equation}
L  \geq L(\delta_c) = \left[ \frac{1-\nu}{1+\nu}\cdot \frac{1}{2\delta_c}\right]^{\frac{1-\nu}{1+\nu}} M_{\nu}^{\frac{2}{1+\nu}}, 
\label{eq_ho}
\end{equation}
then
\begin{equation} 
\label{eq_hotoLip}
f(y) \leq f(x) + \left\langle \nabla f(x), y-x\right\rangle + \frac{L}{2}\|x-y\|_{E}^2+\delta_c, \quad \forall x,y \in Q .
\end{equation}
We assume also that the set $Q$ is bounded with $\max_{x,y \in Q} \|x-y\|_{E} \leq D$. Finally, assume that the value and subgradient of $f$ can be calculated only with some known, but uncontrolled error. Strictly speaking, there exist $\bar{\delta}_1, \bar{\delta}_2 > 0$  such that, for any point $x \in Q$, we can calculate approximations $\bar{f}(x)$ and $\bar{g}(x)$ with $|\bar{f}(x) - f(x)| \leq \bar{\delta}_1$ and $\|\bar{g}(x)-\nabla f(x)\|_{E^{\ast}} \leq \bar{\delta}_2$.

Let us show that, in this example, $f$ can be equipped with inexact first-order oracle based on the pair $(\bar{f}(x),\,\bar{g}(x))$, where $f_{\delta}(x)=\bar{f}(x)-\bar{\delta}_1-\bar{\delta}_2 D$ and $g_{\delta}(x)=\bar{g}(x)$.

Now we prove the first inequality from \eqref{eq_or}
\begin{align}
    f(y) &\geq f(x) + \left\langle \nabla f(x), y-x\right\rangle\\
         &\geq \bar{f}(x) - \bar{\delta}_1 + \left\langle \bar{g}(x), y-x\right\rangle- \bar{\delta}_2 D = f_{\delta}(x) + \left\langle g_{\delta}(x), y-x\right\rangle
\end{align}

Using inequality \eqref{eq_hotoLip} we obtain the second inequality from \eqref{eq_or}, for any $y \in Q$, 
\begin{align}
f(y) &\leq f(x) + \la \nabla f(x), y-x \ra + \frac{L(\delta_c)}{2}\|x-y\|_{E}^2 +\delta_c \notag \\
     &\leq \bar{f}(x) + \bar{\delta}_1 + \la \bar{g}(x), y-x \ra + \la \nabla f(x) - \bar{g}(x), y-x \ra + \frac{L(\delta_c)}{2}\|x-y\|_{E}^2 + \delta_c \notag \\
	   &\leq f_{\delta}(x)  + \la g_{\delta}(x), y-x \ra + \frac{L(\delta_c)}{2}\|x-y\|_{E}^2 + 2\bar{\delta}_1+ 2\bar{\delta}_2 D + \delta_c. \notag
\end{align}
Thus,  $(f_{\delta}(x),\,g_{\delta}(x))$ is an inexact first-order oracle with $\delta_u =  2\bar{\delta}_1+ 2\bar{\delta}_2D$,  $\delta_c$, and $L(\delta_c)$ given by \eqref{eq_ho}.

To construct our algorithm for problem \eqref{eq_pr}, we introduce, as it is usually done, \text{proximal setup} \cite{ben-tal2015lectures}, which consists of choosing a norm $\|\cdot\|_E$, and a {\it prox-function} $d(x)$ which is continuous, convex on $Q$ and
\begin{enumerate}
	\item $d(x)$ is  a continuously differentiable $1$-strongly convex on $Q$ with respect to $\|\cdot\|_E$, i.e., for any $x,y \in \text{rint } Q$, 
  \begin{equation}
d(y)-d(x) -\langle \nabla d(x) ,y-x \rangle \geq \frac12\|y-x\|_E^2.
\label{eq_d}
\end{equation}  
	\item Without loss of generality, we assume that 
    \begin{equation}
    \label{eq_mind}
    \min\limits_{x\in  Q} d(x) = 0.
    \end{equation}
 Then if $\bar{x}=\argmin\limits_{x\in  Q} d(x)$, we get 
\begin{equation}
\label{eq_dcenter}
    d(y)\geq \frac12\|y-\bar{x}\|_E^2, \quad \forall y \in Q.
\end{equation}
\end{enumerate}

The corresponding \textit{Bregman divergence} is defined as $V(x,y)=d(y)-d(x)-\left\langle \nabla d(x), y-x \right\rangle$ and satisfies
\begin{equation}
\label{eq_breg}
V(x,y)\geq \frac{1}{2} \|x-y\|_E^2, \quad \forall x,y \in Q.  
\end{equation}
We use prox-function in so called {\it composite prox-mapping}, which consists in solving auxiliary problem
\begin{equation}
\min_{x \in Q} \left\{\la g,x \ra + d(x) +h(x) 		 \right\} ,
\label{eq_PrMap}
\end{equation}
where $g \in E^{\ast}$ is given. We allow this problem to be solved inexactly in the following sense.
\begin{definition}
Assume that $\delta_{p} >0$, $g \in E^{\ast}$ are given.
We call a point $\tilde{x} = \tilde{x}(g,\delta_{p}) \in \text{rint }  Q$ an {\it inexact composite prox-mapping} iff we can calculate $\tilde{x}$ and there exists $p \in \partial h(\tilde{x})$ s.t. it holds that
\begin{equation}
\left\la g + \nabla d(\tilde{x}) + p, u - \tilde{x} \right\ra \geq - \delta_{p}, \quad \forall u \in  Q.
\label{eq_InPrMap}
\end{equation}
 We denote by
\begin{equation}
\tilde{x} =  \argmin\limits_{x\in  Q} \,^{\delta_{p}}\left\{\la g,x \ra + d(x) +h(x) 		 \right\}.
\label{eq_InPrMap1}
\end{equation}
 one of the possible inexact composite prox-mapping. 
\end{definition}

 Note that if $\tilde{x}$ is an exact solution of \eqref{eq_PrMap}, inequality \eqref{eq_InPrMap} holds with $\delta_{p}=0$ due to first-order optimality condition. 

We also use the following auxiliary fact

\begin{lemma} (Lemma 5.5.1 in \cite{ben-tal2015lectures}) 
Let $F \, : Q \rightarrow \mathbb{R} \bigcup  \{+\infty\}$ be a convex function such that $\Psi(x)=F(x) + d(x)$ is closed and convex on $Q$. Denote $\tilde{x} =\argmin\limits_{x\in Q}\,^{\delta_{p}} \Psi(x).$ Then 
\begin{equation}
    \Psi(y)\geq \Psi(\tilde{x})+ V(\tilde{x},y)-\delta_{p}, \quad \forall y \in Q.
\label{eq_l2}
\end{equation}
Hence, from \eqref{eq_breg}
\begin{equation}
    \Psi(y)\geq \Psi(\tilde{x})-\delta_{p}, \quad \forall y \in Q.
\label{eq_l2b}
\end{equation}
\label{lem2}
\end{lemma}

\section{Universal Intermediate Gradient Method}
In this section, we describe a general scheme of Universal Intermediate Gradient Method (UIGM) and prove general convergence rate. This scheme is based on two sequences $\alpha_{k}$,  $B_{k}$, $k \geq 0$. From now on, we assume that these sequences satisfy, for all $k\geq 0$,
\begin{equation}
\label{eq_ABineq}
0 < \alpha_{k+1} \leq B_{k+1} \leq A_k + \alpha_{k+1},
\end{equation}
where the sequence $A_k$ is defined by recurrence $A_{k+1} = A_k + \alpha_{k+1}$. Particular choice of these two sequences and its consequence for the convergence rate are discussed in the next section.

 For Algorithm 1 we combine Algorithm 2 from \cite{devolder2013intermediate} with Algorithm 1 from \cite{gasnikov2016universal} to get IGM with only one prox-mapping instead two as in \cite{devolder2013intermediate}. After that we improve this method by techniques from \cite{nesterov2015universal} to get UIGM with exact prox-mapping. Last generalization use Lemma 5.5.1 from \cite{ben-tal2015lectures}. As a result we get algorithm that works in wide class of problems, adaptive and don't need to know exact H\"older and Lipschitz constants, uses only one prox-mapping and correctly work with errors of oracle and prox-mapping.

\newpage

\begin{algorithm}[H]
    \caption{Universal Intermediate Gradient Method (UIGM)}
    \label{Universal FGM}
    \begin{algorithmic}[1]
    \Require $\varepsilon >0$ -- desired accuracy, $\delta_u$ -- uncontrolled oracle error,  $\delta_p$ -- prox-mapping error , $L_s$ -- initial guess for the H\"older constant, $\alpha_k$ -- choose by some policy, for example \eqref{eq_apower}.
    \State  Set $\delta_0=\frac{\varepsilon}{4}+\delta_u$, 		\begin{equation}
 		\label{eq_x0}
 			z_0=x_0=\argmin_{x\in Q}\,^{\delta_{p}}d(x),
 		\end{equation}
 	\State Set $i_0=0$  
 	    \State Compute \begin{equation}
    y_{0} =\argmin\limits_{x \in Q}\,^{\delta_{p}} \left\lbrace  d(x) + (2^{i_0} L_s)^{-1} \left[ \langle g_{\delta_0}(x_0), x-x_{0}\rangle+ h(x)\right]\right\rbrace.
    \label{eq_y0}
    \end{equation}
    \State If
    \begin{equation}
	    \begin{aligned}
	    f_{\delta_0}\left(y_{0}\right)\leq f_{\delta_0}\left(x_{0}\right)+\left\langle g_{\delta_0}\left(x_{0}\right), y_{0}-x_{0}\right\rangle + \frac{2^{i_0}L_s}{2} \Vert y_{0}-x_{0}\Vert_E^2+\delta_0,
	    \label{eq_d0}
	    \end{aligned}
	    \end{equation}
        go to Step 5. Otherwise, set $i_0=i_0+1$ and go back to Step 3.
    \State Define $L_0= 2^{i_0}L_s$, $\alpha_0 =B_0=A_0=(L_0)^{-1}$.
    \For{$ k = 1,\dots$}
        \State  Set $i_k=0$.
        \State   Set $L_{k} = 2^{i_k}L_{k-1}$ and $\alpha_k=\alpha(L_k)$ by some policy, for example \eqref{eq_apower},
    		\begin{align}
	     		&B_{k}=\alpha^2_{k}L_{k}\label{eq_B},\\
	    		&\delta_{k}=\frac{\alpha_{k}}{B_{k}}\frac{\varepsilon}{4}+\delta_u \label{eq_dc}\\
	     		&x_{k} = \frac{\alpha_{k}}{B_{k}} z_{k-1} +  \frac{B_{k}-\alpha_{k}}{B_{k}}  y_{k-1}. \label{eq_xk1alg}\\
	    			z_{k} = \argmin_{x\in Q}\,^{\delta_{p}}&\left\lbrace  d(x)+ 				\sum\limits_{j=0}^{k}\alpha_j\left[\langle g_{\delta_j}(x_j), x-x_j\rangle+h(x) \right]   \right\rbrace ,
	    		\label{eq_z}\\
				&w_{k}=\frac{\alpha_{k}}{B_{k}}z_{k}+\frac{B_{k}-\alpha_{k}}{B_{k}}y_{k-1}.\label{eq_w}
			\end{align}
    \State   If  
        \begin{equation}
	    f_{\delta_{k}}(w_{k})\leq f_{\delta_{k}}(x_{k})+\delta_{k}+\langle g_{\delta_{k}}(x_{k}), w_{k}-x_{k}\rangle+\frac{L_{k}}{2}\Vert w_{k}-x_{k}\Vert_E^2.
	    \label{eq_dm}
	    \end{equation} 
	    go to Step 10. Otherwise, set $i_k=i_k+1$ and go back to Step 8.
	\State Set 
	\begin{align}
    &A_{k}=A_{k-1}+\alpha_{k},  \label{eq_A}\\
	\label{eq_y}
	&y_{k}=\dfrac{B_{k}}{A_{k}}w_{k}+\dfrac{A_{k}-B_{k}}{A_{k}}y_{k-1}.
	\end{align}
    \EndFor
    \end{algorithmic}
\end{algorithm}

\newpage
 The next theorem gives an upper bound for $A_k F(y_k)$. Its proof is an adaptation of the proof of Lemma 1 in \cite{devolder2013intermediate} and Theorem 3 in \cite{nesterov2015universal}.
\begin{theorem}
\label{ufgm0}
Let $f$ be a convex function with inexact first-order oracle, the dependence $L(\delta_c)$ being given by \eqref{eq_ho}. Then all iterations of UIGM are well defined and, for all $k \geq 0$ we have
\begin{equation}
    \label{ufgm1}
    A_{k}F(y_k) - E_k \leq \Psi_{k}^{\ast}, 
\end{equation}
where $E_k=2\left( \sum \limits_{j=0}^k B_j \right)\delta_u+(2k+1)\delta_{p}+ A_k\frac{\varepsilon}{2}$,
\begin{equation}
\label{ufgm2}
    \Psi_{k}^{\ast}= \min_{x\in Q}\left\lbrace \Psi_k (x) = d(x)+ \sum\limits_{j=0}^k\alpha_j\left[f_{\delta_j}(x_{j})+\langle g_{\delta_j}(x_{j}), x-x_j\rangle+h(x)\right] \right\rbrace .
\end{equation}

\end{theorem}
\begin{proof}

Let us prove first, that the "line-search" process of steps 6--9 is finite. By \eqref{eq_ho}, \eqref{eq_hotoLip}, if $2^{i_k} L_{k-1} \geq L\left(\frac{\alpha_{k}}{B_{k}}\frac{\varepsilon}{4}\right)$, from \eqref{eq_dm} and \eqref{eq_fd}, we get
\begin{equation*}
	    f_{\delta_{k}}(w_{k}) - \delta_{k} \leq f(w_{k})\leq f_{\delta_{k}}(x_{k})+\langle g_{\delta_{k}}(x_{k}), w_{k}-x_{k}\rangle+\frac{2^{i_{k}}L_{k-1}}{2}\Vert w_{k}-x_{k}\Vert_E^2+\delta_{k}
	    \end{equation*}
and the stopping criterion in the inner cycle holds. Thus, we need to show that 
\begin{equation}
\label{eq_LUpDown}
    2^{i_k}L_{k-1}\geq \left[\dfrac{\alpha_{k}}{ B_{k}}\varepsilon \right]^{-\frac{1-\nu}{1+\nu}}M_\nu^{\frac{2}{1+\nu}}
\end{equation}
for $i_k$ large enough. Indeed,
\begin{equation*}
    2^{i_k}L_{k-1} \left[\dfrac{\alpha_{k}}{ B_{k}}\right]^{\frac{1-\nu}{1+\nu}} \overset{\eqref{eq_B}}{=} \frac{B_{k}}{\alpha_{k}^2}\left[\dfrac{\alpha_{k}}{B_{k}}\right]^{\frac{1-\nu}{1+\nu}}=  \left[ \frac{B_{k}}{\alpha_{k}} \right]^{\frac{2\nu}{1+\nu}}\frac{1}{\alpha_{k}}  \overset{\eqref{eq_ABineq}}{\geq} \frac{1}{\alpha_{k}}.
\end{equation*}
It remains to prove that $\alpha_{k}\rightarrow 0$ as $i_k \rightarrow \infty $.
\begin{equation*}
    \alpha_{k}^2=\frac{B_{k}}{2^{i_k}L_{k-1}}
    \overset{\eqref{eq_ABineq}}{\leq} \frac{A_{k}}{2^{i_k}L_{k-1}}
    \overset{\eqref{eq_A}}{=} \frac{A_{k-1}+\alpha_{k}}{2^{i_k}L_{k-1}},
\end{equation*}
\begin{equation}
\Rightarrow \qquad \alpha_{k}^2 -\frac{\alpha_{k}}{2^{i_k}L_{k-1}}-\frac{A_{k-1}}{2^{i_k}L_{k-1}}\leq 0.
\label{eq_kvad}
\end{equation}
Thus,  $\alpha_{k}\in \left[\alpha^{-}_{k},\alpha^{+}_{k}\right]$, where $\alpha^{-}_{k}$ and $\alpha^{+}_{k}$ are the solutions of
\begin{equation*}
\alpha_{k}^2 -\frac{\alpha_{k}}{2^{i_k}L_{k-1}}-\frac{A_{k-1}}{2^{i_k}L_{k-1}}=0.
\end{equation*}
The solutions are
\begin{align*}
&\alpha^{-}_{k}=\frac{1}{2^{i_k+1}L_{k-1}}-\left(\frac{1}{4^{i_k+1}L_{k-1}^2}+\frac{A_{k-1}}{2^{i_k}L_{k-1}}  \right)^{1/2},\\
&\alpha^{+}_{k}=\frac{1}{2^{i_k+1}L_{k-1}}+\left(\frac{1}{4^{i_k+1}L_{k-1}^2}+\frac{A_{k-1}}{2^{i_k}L_{k-1}}  \right)^{1/2}.
\end{align*}
Now from \eqref{eq_kvad} we have that $\alpha^{-}_{k}\leq\alpha_{k}\leq\alpha^{+}_{k}$. From $\alpha^{-}_{k} \rightarrow 0$, $\alpha^{+}_{k} \rightarrow 0$
as $i_k \rightarrow \infty$ we get $\alpha_{k} \rightarrow 0$. 

 Let us prove relation \eqref{ufgm1}. For $k=0$: 
\begin{align*}
\Psi_{0}^{\ast}&\overset{\eqref{ufgm2}}{=}\min_{x\in Q}\left\lbrace  d(x)+ \alpha_0 f_{\delta_0}(x_0) +\alpha_0\langle g_{\delta_0}, x-x_0\rangle +\alpha_0 h(x)\right\rbrace \\
&\overset{\eqref{eq_l2b}, \eqref{eq_y0}}{\geq}d(y_0)+\alpha_0 f_{\delta_0}(x_0)  +\alpha_0\langle g_{\delta_{0}}(x_{0}), y_0-x_0\rangle +\alpha_0 h(y_0)-\delta_{p}\\
&\overset{\eqref{eq_dcenter},\eqref{eq_x0}}{\geq} \alpha_0 \left( \frac{1}{2 \alpha_0}\|y_0-x_0 \|^2_E+f_{\delta_0}(x_0)  + \langle g_{\delta_{0}}(x_{0}), y_0-x_0\rangle + h(y_0) \right)-\delta_{p}\\
&=\alpha_0 \left( \frac{2^{i_0}L_s}{2}\|y_0-x_0 \|^2_E+f_{\delta_0}(x_0)  + \langle g_{\delta_{0}}(x_{0}), y_0-x_0\rangle + h(y_0) \right)-\delta_{p}\\
&\overset{\eqref{eq_d0}}{\geq} \alpha_0 \left(f_{\delta_0}(y_0) - \frac{\varepsilon}{4}-\delta_u + h(y_0)  \right)-\delta_{p}\\
&\overset{\eqref{eq_fd}}{\geq} \alpha_0 \left(f(y_0) - \frac{\varepsilon}{2}-2\delta_u + h(y_0)  \right)-\delta_{p} = A_0 F(y_0) - E_0.
\end{align*}

Assume that \eqref{ufgm1} is valid for certain $k-1\geq 0$. We now prove that it holds for $k$. 

\begin{align*}
\Psi_{k}^{\ast}&\overset{\eqref{ufgm2}}{=}\min\limits_{x \in Q} \Psi_{k}(x)\overset{\eqref{eq_l2b}, \eqref{eq_z}}{\geq}\Psi_{k}(z_{k})- \delta_{p}\\
&\overset{\eqref{ufgm2}}{=}\Psi_{k-1}(z_{k})+\alpha_{k}\left[f_{\delta_{k}}(x_{k})+\langle g_{\delta_{k}}(x_{k}), z_{k}-x_{k} \rangle + h(z_{k}) \right]-\delta_{p}\\ 
&\overset{\eqref{eq_l2}}{\geq} \Psi_{k-1}(z_{k-1})+V(z_{k-1}, z_{k})-2\delta_{p}+\alpha_{k}\left[f_{\delta_{k}}(x_{k})+\langle g_{\delta_{k}}(x_{k}), z_{k}-x_{k} \rangle + h(z_{k}) \right]\\ 
&\overset{\eqref{eq_breg}}{\geq} \Psi_{k-1}^{\ast}+\frac{1}{2}\|z_{k} - z_{k-1}\|_E-2\delta_{p}+\alpha_{k}\left[f_{\delta_{k}}(x_{k})+\langle g_{\delta_{k}}(x_{k}), z_{k}-x_{k} \rangle + h(z_{k}) \right]\\
&\overset{\eqref{ufgm1}}{\geq} A_{k-1}  F(y_{k-1})-E_{k-1}+\frac{1}{2}\|z_{k} - z_{k-1}\|_E -2\delta_{p}\\
&+\alpha_{k}\left[f_{\delta_{k}}(x_{k})+\langle g_{\delta_{k}}(x_{k}), z_{k}-x_{k} \rangle + h(z_{k}) \right]\\ 
&= \left(A_{k}-B_{k} \right) F(y_{k-1})-E_{k-1} +\frac{1}{2}\|z_{k} - z_{k-1}\|_E-2\delta_{p}+ \left(B_{k}-\alpha_{k} \right) f(y_{k-1})\\
& +\left(B_{k}-\alpha_{k} \right) h(y_{k-1})+ \alpha_{k} h(z_{k}) +\alpha_{k}\left[f_{\delta_{k}}(x_{k})+\langle g_{\delta_{k}}(x_{k}), z_{k}-x_{k} \rangle  \right]\\
&\overset{\eqref{eq_w}}{\geq} \left(A_{k}-B_{k} \right) F(y_{k-1})-E_{k-1}+B_{k} h(w_{k}) +\frac{1}{2}\|z_{k} - z_{k-1}\|_E-2\delta_{p}\\
&+ \left(B_{k}-\alpha_{k} \right) f(y_{k-1}) +\alpha_{k}\left[f_{\delta_{k}}(x_{k})+\langle g_{\delta_{k}}(x_{k}), z_{k}-x_{k} \rangle  \right]\\ 
\end{align*}
\begin{align*}
&\geq \left(A_{k}-B_{k} \right) F(y_{k-1})-E_{k-1}+B_{k} h(w_{k}) +\frac{1}{2}\|z_{k} - z_{k-1}\|_E-2\delta_{p}\\
&+  \left(B_{k}-\alpha_{k} \right) \left( f_{\delta_{k}}(x_{k})+\langle  g_{\delta_{k}}(x_{k}), y_{k-1}-x_{k} \rangle    \right) +\alpha_{k}\left[f_{\delta_{k}}(x_{k})+\langle g_{\delta_{k}}(x_{k}), z_{k}-x_{k} \rangle + h(z_{k})  \right]\\ 
&= \left(A_{k}-B_{k} \right) F(y_{k-1})-E_{k-1}+B_{k} h(w_{k}) +\frac{1}{2}\|z_{k} - z_{k-1}\|_E +B_{k}f_{\delta_{k}}(x_{k}) \\
&+  \langle  g_{\delta_{k}}(x_{k}), \left(B_{k}-\alpha_{k} \right)  (y_{k-1}-x_{k}) +\alpha_{k} (z_{k}-x_{k}) \rangle -2\delta_{p}. 
\end{align*}
From \eqref{eq_xk1alg}, we have
\[\left(B_{k}-\alpha_{k} \right) \left(  y_{k-1}-x_{k} \right) 
+\alpha_{k}\left(z_{k}-x_{k} \right)=\alpha_{k}(z_{k}-z_{k-1}).
\]
Therefore,
\begin{align*}
\Psi_{k}^{\ast}&\geq \left(A_{k}-B_{k} \right) F(y_{k-1})-E_{k-1}+B_{k} h(w_{k}) -2\delta_{p}\\
&+B_{k}f_{\delta_{k}}(x_{k}) + \alpha_{k}\langle  g_{\delta_{k}}(x_{k}), (z_{k}-z_{k-1}) \rangle +\frac{1}{2}\|z_{k} - z_{k-1}\|_E  \\  
&= \left(A_{k}-B_{k} \right) F(y_{k-1})-E_{k-1}+B_{k} h(w_{k})-2\delta_{p}\\
&+B_{k}\left[f_{\delta_{k}}(x_{k}) +  \frac{\alpha_{k}}{B_{k}}\langle g_{\delta_{k}}(x_{k}), z_{k}-z_{k-1} \rangle + \frac{1}{2B_{k}} \|z_{k}-z_{k-1}\|^2_{E} \right] .
\end{align*}
As $B_{k}= 2^{i_k}L_{k-1} \alpha^2_{k}$, we have $\frac{1}{B_{k}}= 2^{i_k}L_{k-1}\frac{\alpha_{k}^2}{B_{k}^2} $ and, therefore,

\begin{align*}
\Psi_{k}^{\ast} &\geq \left(A_{k}-B_{k} \right) F(y_{k-1})-E_{k-1}+B_{k} h(w_{k})-2\delta_{p}\\
&+B_{k}\left[f_{\delta_{k}}(x_{k})+ \frac{\alpha_{k}}{B_{k}}\langle g_{\delta_{k}}(x_{k}), z_{k}-z_{k-1} \rangle + \frac{ 2^{i_k}L_{k-1}\alpha^2_{k}}{2B^2_{k}} \|z_{k}-z_{k-1}\|^2_{E}  \right].
\end{align*}
But
\[\frac{\alpha_{k}}{B_{k}}(z_{k}-z_{k-1})\overset{\eqref{eq_xk1alg}, \eqref{eq_w}}{=} w_{k}-x_{k},
\]
and we obtain
\begin{align*}
\Psi_{k}^{\ast} &\geq \left(A_{k}-B_{k} \right) F(y_{k-1})-E_{k-1}+B_{k} h(w_{k})-2\delta_{p}\\
&+B_{k}\left[f_{\delta_{k}}(x_{k})+ \langle g_{\delta_{k}}(x_{k}), w_{k}-x_{k} \rangle + \frac{ 2^{i_k}L_{k-1}}{2} \|w_{k}-x_{k}\|^2_{E}   \right]\\ 
&\overset{\eqref{eq_dm}}{\geq} \left(A_{k}-B_{k} \right) F(y_{k-1})-E_{k-1}+B_{k} h(w_{k}) -2\delta_{p}+B_{k}\left[f_{\delta_{k}}(w_{k})-\frac{\alpha_{k}}{B_{k}}\frac{\varepsilon}{4}-\delta_u \right]\\
&\overset{\eqref{eq_fd}}{\geq} \left(A_{k}-B_{k} \right) F(y_{k-1})-E_{k-1}+B_{k} h(w_{k})-2\delta_{p}+B_{k}\left[f(w_{k})-\frac{\alpha_{k}}{B_{k}}\frac{\varepsilon}{2}-2\delta_u \right]\\
&\overset{\eqref{eq_y}}{\geq} A_{k} F(y_{k})-E_{k-1} - B_{k}\left[\frac{\alpha_{k}}{B_{k}}\frac{\varepsilon}{2}+2\delta_u \right]-2\delta_{p}\\
&=A_{k} F(y_{k})-E_{k}.
\end{align*}
\end{proof}

 We are in position to establish the relation between the rate of growth of $\left\lbrace A_k \right\rbrace_{k=0}^{+\infty}$ with rate of convergence of UIGM. The proof of the next result is an adaptation of Theorem 2 in \cite{devolder2013intermediate}.
\begin{corollary}
\label{Cor:DeltaFRate}
Let $f$ be a convex function with inexact first-order oracle, the dependence $L(\delta_c)$ being given by \eqref{eq_ho}. Then all iterations of UIGM are well defined and, for all $k \geq 0$, we have

\begin{equation}
\label{eq_UIGM_A_B}
    F(y_k)-F^{\ast} \leq \frac{d(x^{\ast})}{A_k} + \frac{2\delta_u}{A_k}\sum \limits_{j=0}^k B_j+\frac{(2k+1)\delta_{p}}{A_k}+\frac{\varepsilon}{2}.
\end{equation}

\end{corollary}

\begin{proof}

\begin{align*}
\Psi_{k}^{\ast} &= \min_{x\in Q}\left\lbrace d(x)+ \sum\limits_{j=0}^k\alpha_j\left[f_{\delta_{j}}(x_{j})+\langle g_{\delta_{j}}(x_{j}), x-x_j\rangle+h(x)\right] \right\rbrace\\
&\leq d(x^{\ast})+ \sum\limits_{j=0}^k\alpha_j\left[f_{\delta_{j}}(x_{j})+\langle g_{\delta_{j}}(x_{j}), x^{\ast}-x_j\rangle+h(x^{\ast})\right]\\
& \overset{\eqref{eq_or}}{\leq} d(x^{\ast})+ \sum\limits_{j=0}^k\alpha_j\left[f(x^{\ast})+h(x^{\ast})\right] = d(x^{\ast})+ A_{k}F(x^{\ast}).
\end{align*}
 By \eqref{ufgm1}, we have $A_k F(y_k)-E_k\leq   d(x^{\ast})+ A_{k}F(x^{\ast})$ and so
\begin{align*}
    F(y_k)-F^{\ast} &\leq \frac{d(x^{\ast})}{A_k}+\frac{E_k}{A_k} \leq \frac{d(x^{\ast})}{A_k} + \frac{2\delta_u}{A_k}\sum \limits_{j=0}^k B_j+\frac{(2k+1)\delta_{p}}{A_k}+ \frac{\varepsilon}{2}.
\end{align*}
\end{proof}

Similarly as UFGM \cite{nesterov2015universal}, UIGM can be equipped with an implementable stopping criterion.  Assume that we know an upper bound $D$ for the distance to the solution from the starting point $V(x_0,x^{\ast})=d(x^{\ast})\leq D$. Denote $l_k^{pd}(x)= \sum\limits_{j=0}^k\alpha_j\left[f_{\delta_{j}}(x_{j})+\langle g_{\delta_{j}}(x_{j}), x-x_j\rangle\right]$ and 
\begin{equation}
\begin{aligned}
\label{eq_lpd}
\tilde{F}_k&=\min\limits_{x \in Q} \left\{\frac{1}{A_k} l_k^{pd}(x)+h(x) :  d(x)\leq D \right\}\\
&=\min\limits_{x \in Q} \max\limits_{\beta\geq 0} \left\{\frac{1}{A_k} l_k^{pd}(x)+h(x) +\beta( d(x)- D) \right\}\\
&=\max\limits_{\beta \geq 0} \min\limits_{x \in Q} \left\{ \frac{1}{A_k}l_k^{pd}(x)+h(x)+\beta(d(x)-D) \right\}\\
&\overset{\beta=1/A_k}{\geq} \frac{1}{A_k}\Psi_k^{\ast}-\frac{1}{A_k}D.
\end{aligned}
\end{equation}
 Note that by the first inequality from \eqref{eq_or} we get $\tilde{F}_k\leq F^{\ast}$. Then 
\begin{equation*}
    F(y_k)-F^{\ast}\leq F(y_k)- \tilde{F}_k \overset{\eqref{ufgm1}, \eqref{eq_lpd}}{\leq} \frac{D}{A_k} + \frac{E_k}{A_k}
\end{equation*}

Thus, we can use stopping criterion 
\begin{equation}
\label{eq_pdstop}
F(y_k)-\tilde{F}_k\leq \varepsilon + \frac{2\delta_u}{A_k}\sum \limits_{j=0}^k B_j+ \frac{(2k+1)\delta_{p}}{A_k},
\end{equation}
which ensures
\begin{equation}
F(y_k)-F^{\ast}\leq \varepsilon+\frac{2\delta_u}{A_k}\sum \limits_{j=0}^k B_j+ \frac{(2k+1)\delta_{p}}{A_k},
\label{eq_error}
\end{equation}
as far as
\begin{equation}
\label{eq_Akup}
A_k\geq \frac{2D}{\varepsilon}.
\end{equation}

At the end we get an upper bound of the total number of oracle calls for UIGM with stopping criterion \eqref{eq_pdstop} to get an approximate solution of problem \eqref{eq_pr} satisfying \eqref{eq_error}. 

Denote by $N(k)$ the total number of oracle calls after $k$ iterations  (without $0$ iteration). We don't take $0$ into account because it is some constant that depends on initial guess $L_s$. At each iteration we call oracle at points $x_{m}$ and $w_{m}$ and do it $(i_m+1)$ times. Then total number of oracle calls per iteration equal to $2(i_m+1)$. Note that  $L_{m}=2^{i_m}L_{m-1}$. Therefore, $i_m=\log_2 \frac{L_{m}}{L_{m-1}}$. Hence,

\begin{equation}
\label{eq_NOracle}
\begin{aligned}
N(k)&=\sum\limits_{m= 1}^k 2(i_m+ 1)=\sum\limits_{m=1}^k 2(\log_2 \frac{L_{m}}{L_{m-1}}+ 1)=\sum\limits_{m=1}^k \left[2+2(\log_2 L_{m}- \log_2 L_{m-1})\right]\\
&= 2k+2\log_2 L_{k}-2\log_2 L_0.
\end{aligned}
\end{equation}

Note that
\begin{equation*}
\frac{B_{k}}{2\alpha_{k}^2}\overset{\eqref{eq_B}}{=}2^{i_k}L_{k-1}=L_{k}\overset{\eqref{eq_ho},\eqref{eq_LUpDown}}{\leq} \left[\frac{1-\nu}{1+\nu}\dfrac{1}{\dfrac{\alpha_{k}}{ B_{k}}\varepsilon} \right]^{\frac{1-\nu}{1+\nu}}M_\nu^{\frac{2}{1+\nu}}=\left[\dfrac{B_{k}(1-\nu)}{\varepsilon\alpha_{k}(1+\nu)} \right]^{\frac{1-\nu}{1+\nu}}M_\nu^{\frac{2}{1+\nu}},
\end{equation*}
\begin{equation*}
\Rightarrow \quad \alpha_{k}^{\frac{-1-3\nu}{1+\nu}}\leq 2B_{k}^{\frac{-2\nu}{1+\nu}}\left[\dfrac{M_\nu^{\frac{2}{1+\nu}}}{\varepsilon^{\frac{1-\nu}{1+\nu}}}\left(\frac{1-\nu}{1+\nu}\right)^\frac{1-\nu}{1+\nu}\right].
\end{equation*}
Therefore,
\begin{equation*}
L_{k}\leq \frac{B_{k}}{2} \left(2B_{k}^{\frac{-2\nu}{1+\nu}}\left[\dfrac{M_\nu^{\frac{2}{1+\nu}}}{\varepsilon^{\frac{1-\nu}{1+\nu}}} \left(\frac{1-\nu}{1+\nu}\right)^\frac{1-\nu}{1+\nu}\right]\right)^{\frac{2(1+\nu)}{1+3\nu}}\leq  2^{\frac{1-\nu}{1+3\nu}} A_{k}^{\frac{1-\nu}{1+3\nu}}\left[\dfrac{M_\nu^{\frac{4}{1+3\nu}}}{\varepsilon^{\frac{2-2\nu}{1+3\nu}}} \left(\frac{1-\nu}{1+\nu}\right)^\frac{2-2\nu}{1+3\nu}\right].
\end{equation*}
Note that \eqref{eq_error} holds if \eqref{eq_Akup} holds. Thus, we can assume that, during the iterations,
\begin{equation*}
A_k\leq \frac{2D}{\varepsilon}, \quad k \geq 0.
\end{equation*}
Hence,
\begin{equation*}
L_{k+1}\leq  2^{\frac{1-\nu}{1+3\nu}} \left(\frac{2D}{\varepsilon}\right)^{\frac{1-\nu}{1+3\nu}}\left[\dfrac{M_\nu^{\frac{4}{1+3\nu}}}{\varepsilon^{\frac{2-2\nu}{1+3\nu}}} \left(\frac{1-\nu}{1+\nu}\right)^\frac{2-2\nu}{1+3\nu}\right].
\end{equation*}
Substituting this estimate in the expression \eqref{eq_NOracle}, we obtain that on average UIGM has approximately  two calls of oracle per iteration.

\section{Power policy}
In this section, we present particular choice of the two sequences of coefficients $\{\alpha_{k}\}_{k\geq 0}$ and $\{B_{k}\}_{k\geq 0}$. As it was done in \cite{devolder2013intermediate}, these sequences depend on a parameter $p \in [1,2]$. In our case, the value $p=1$ corresponds to Universal Dual Gradient Method, and the value $p=2$ corresponds to Universal Fast Gradient Method. For the smooth case, namely $\nu=1$, the method in \cite{devolder2013intermediate} has convergence rate
\begin{equation*}
     F(y_k)-F^{\ast} \leq \Theta\left(\frac{d(x^{\ast})}{k^p}\right) + \Theta\left(k^{p-1} \delta_u \right),
 \end{equation*}
 where $p \in [1,2]$. Our goal to obtain convergence rate for the whole segment $\nu \in [0,1]$ and get the above rate of convergence as a special case.

Given a value $p \in [1,2]$, we choose sequences $\{\alpha_{k}\}_{k\geq 0}$ and $\{B_{k}\}_{k\geq 0}$ to be given by
\begin{equation}
\label{eq_apower}
\alpha_{k}=\frac{\left(\frac{k+2p}{2p}\right)^{p-1}}{ 2^{i_k}}L_{k-1},  \quad k\geq 0  
\end{equation}
and, in accordance to \eqref{eq_B},
\begin{equation}
\label{eq_bpower}
B_{k}=\frac{\left(\frac{k+2p}{2p}\right)^{2p-2}}{ 2^{i_k}}L_{k-1}, \quad k\geq 0.   
\end{equation}
 Now we should prove, that power policy can be used in UIGM.
\begin{lemma}
\label{power_ok}
Assume that $f$ is a convex function with inexact first-order oracle. Then, the sequences $\{\alpha_{k}\}_{k\geq 0}$ and $\{B_{k}\}_{k\geq 0}$ given in \eqref{eq_apower} and \eqref{eq_bpower}, respectively,  satisfy \eqref{eq_ABineq}.
\end{lemma}
\begin{proof}
From \eqref{eq_apower} we get that $\alpha_{k} >0$ for $k\geq 0$. To prove that $\alpha_{k} \leq B_{k}$ for $k\geq 0$, we use \eqref{eq_apower}, \eqref{eq_bpower} and that $p \in [0,1]$
\begin{align*}
\alpha_{k}=\frac{\left(\frac{k+2p}{2p}\right)^{p-1}}{2^{i_k}L_{k-1}}=\frac{\left(\frac{k}{2p}+1\right)^{p-1}}{2^{i_k}L_{k-1}}\leq \frac{\left(\frac{k}{2p}+1\right)^{2p-2}}{2^{i_k}L_{k-1}}=B_{k}.
\end{align*}
Proof of $A_{k} \geq B_{k}$.
For $k=0$ it's correct by definition.
Assume that $A_{k} \geq B_{k}$ is valid for certain $k-1 \geq 0$. We now prove that  it holds for $k$.
For $m \in [0,1]$ and $x,y\geq 0$ function  $f(x,y)=x^{m}+y^{m}-(x+y)^m$ has minimal value  greater or equal to $0$, hence,
\begin{align*}
x^{m}+y^{m}&-(x+y)^m\geq 0,\\
\Rightarrow \quad x^{m}+y^{m}&\geq (x+y)^m,\\
\Rightarrow \quad x^{p-1}+y^{p-1} &\overset{m=p-1}{\geq} (x+y)^{p-1}, \quad p\in [1,2],\\
\Rightarrow \quad \left((k-1+2p)^2\right)^{p-1}+\left(2(k+2p)\right)^{p-1} &\geq \left((k-1+2p)^2+2(k+2p)\right)^{p-1},\\
\Rightarrow \quad \left(k-1+2p\right)^{2(p-1)}+\left(2p(k+2p)\right)^{p-1} &\geq \left((k-1+2p)^2+2(k-1+2p)+2\right)^{p-1},\\
\Rightarrow \quad \left(k-1+2p\right)^{2(p-1)}+\left(2p(k+2p)\right)^{p-1} &\geq \left(k+2p\right)^{2(p-1)},\\
\Rightarrow \quad \left(\frac{k-1+2p}{2p}\right)^{2(p-1)}+\left(\frac{k+2p}{2p}\right)^{p-1} &\geq \left(\frac{k+2p}{2p}\right)^{2(p-1)},\\
\Rightarrow \quad \left(\frac{k-1+2p}{2p}\right)^{2(p-1)} &\geq \left(\frac{k+2p}{2p}\right)^{2(p-1)}-\left(\frac{k+2p}{2p}\right)^{p-1},\\
\Rightarrow \quad \frac{\left(\frac{k-1+2p}{2p}\right)^{2(p-1)}}{L_{k-1}} &\geq \frac{\left(\frac{k+2p}{2p}\right)^{p-1}\left(\left(\frac{k+2p}{2p}\right)^{ p-1}-1\right)}{L_{k-1}},\\
\Rightarrow \quad \frac{\left(\frac{{k-1}+2p}{2p}\right)^{2(p-1)}}{L_{k-1}} &\geq \frac{\left(\frac{k+2p}{2p}\right)^{p-1}\left(\left(\frac{k+2p}{2p}\right)^{ p-1}-1\right)}{2^{i_{k}}L_{k-1}},\\
\Rightarrow \quad \frac{\left(\frac{k-1+2p}{2p}\right)^{2(p-1)}}{L_{k-1}}+\frac{\left(\frac{k+2p}{2p}\right)^{p-1}}{2^{i_{k}}L_{k-1}} &\geq \frac{\left(\frac{k+2p}{2p}\right)^{2(p-1)}}{2^{i_{k}}L_{k-1}},\\
\Rightarrow \quad B_{k-1}+\alpha_{k} &\overset{\eqref{eq_apower},\eqref{eq_bpower}}{\geq} B_{k},\\
\Rightarrow \quad A_{k-1}+\alpha_{k} &\geq B_{k},\\
\Rightarrow \quad A_{k}&\overset{\eqref{eq_A}}{\geq} B_{k}.
\end{align*}
\end{proof}

 Now we can obtain the rate of growth of $\left\lbrace A_k \right\rbrace_{k}=0^{+\infty}$. Combining this rate with Corollary \ref{Cor:DeltaFRate}, we get the explicit rate of convergence of UIGM under the power policy \eqref{eq_apower}. 
\begin{theorem}
Assume that $f$ is a convex function with inexact first-order oracle, the dependence $L(\delta_c)$ being given by \eqref{eq_ho}. Then, for the sequences \eqref{eq_apower} and \eqref{eq_bpower}, for all $k \geq 0$, 
\begin{equation}
\label{eq_fek}
    F(y_k)-F^{\ast} \leq \inf\limits_{\nu\in[0,1]}\left(\frac{16 M_\nu^{\frac{2}{1+\nu}}d(x^{\ast})}{\varepsilon^{\frac{1-\nu}{1+\nu}}(k+2)^{\frac{2p\nu-\nu+1}{1+\nu}}} +\frac{32M_\nu^{\frac{2}{1+\nu}}}{\varepsilon^{\frac{1-\nu}{1+\nu}}(k+2)^{\frac{2\nu(p-1)}{1+\nu}}} \delta_{p}\right)+ 4k^{p-1}\delta_u   + \frac{\varepsilon}{2}.
\end{equation}
\end{theorem}

\begin{proof}
The proof is divided into three steps. First, we prove a lower bound for $\alpha_m$ and $A_m$. Then, we prove upper bound for $B_m$. Finally, we use these bounds in Corollary \ref{Cor:DeltaFRate} and obtain \eqref{eq_fek}.

\underline{Lower bound for $\alpha_m$ and $A_m$.}
Since the inner cycle of UIGM for sure ends when $2^{i_m}L_{m-1} > L(\delta_{m})$, we have $2^{i_m}L_{m-1} \leq 2 L(\delta_{m})$. Hence, 
\begin{equation*}
    2^{i_m}L_{m-1} \overset{\eqref{eq_ho},\eqref{eq_LUpDown}}{\leq} 2 \left[\dfrac{\alpha_{m}}{ B_{m}}\varepsilon \right]^{-\frac{1-\nu}{1+\nu}}M_\nu^{\frac{2}{1+\nu}}\leq 2 \left[\dfrac{\left(\frac{m+2p}{2p}\right)^{p-1}}{\varepsilon} \right]^{\frac{1-\nu}{1+\nu}}M_\nu^{\frac{2}{1+\nu}}, 
\end{equation*}
\begin{equation*}
\Rightarrow \quad \alpha_{m}=\frac{\left(\frac{m+2p}{2p}\right)^{p-1}}{2^{i_m}L_{m-1}}\geq \frac{\left(\frac{m+2p}{2p}\right)^{\frac{2p\nu-2\nu}{1+\nu}}\varepsilon^{\frac{1-\nu}{1+\nu}}}{2 M_\nu^{\frac{2}{1+\nu}}},
\end{equation*}

\begin{equation*}
 \Rightarrow \quad    A_k = \sum\limits_{m=0}^{k} \alpha_m\geq \sum\limits_{m=0}^{k} 
    \frac{\left(\frac{m+2p}{2p}\right)^{\frac{2p\nu-2\nu}{1+\nu}}\varepsilon^{\frac{1-\nu}{1+\nu}}}{2 M_\nu^{\frac{2}{1+\nu}}}\geq 
    \frac{\varepsilon^{\frac{1-\nu}{1+\nu}}}{2 M_\nu^{\frac{2}{1+\nu}}} \sum\limits_{m=0}^{k} \left(\frac{m+2p}{2p}\right)^{\frac{2p\nu-2\nu}{1+\nu}}.
\end{equation*}
Since
\begin{equation*}
\begin{aligned}
    &\sum\limits_{m=0}^{k} \left(\frac{m+2p}{2p}\right)^{\frac{2p\nu-2\nu}{1+\nu}}\geq \int\limits_{0}^{k} \left(\frac{x+2p}{2p}\right)^{\frac{2p\nu-2\nu}{1+\nu}}dx+\alpha_0\\ &\geq\frac{2p(1+\nu)}{2p\nu-\nu+1}\left(\frac{k+2p}{2p}\right)^{\frac{2p\nu-\nu+1}{1+\nu}}\geq 2\left(\frac{k+2p}{2p}\right)^{\frac{2p\nu-\nu+1}{1+\nu}},
    \end{aligned}
\end{equation*}
we have
\begin{equation}
\label{eq_suma}
    A_k = \sum\limits_{m=0}^{k} \alpha_m
\geq \frac{\varepsilon^{\frac{1-\nu}{1+\nu}}}{ M_\nu^{\frac{2}{1+\nu}}}\left(\frac{k+2p}{2p}\right)^{\frac{2p\nu-\nu+1}{1+\nu}}\geq \frac{\varepsilon^{\frac{1-\nu}{1+\nu}}}{ M_\nu^{\frac{2}{1+\nu}}}\left(\frac{k+2}{4}\right)^{\frac{2p\nu-\nu+1}{1+\nu}}.
\end{equation}

\underline{Upper bound for $B_m$.}
\begin{equation*}
B_m\overset{\eqref{eq_apower},\eqref{eq_bpower}}{=}\left(\frac{m+2p}{2p}\right)^{p-1} \alpha_m,
\end{equation*}
\begin{equation*}
    \sum\limits_{m=0}^{k} B_m = \sum\limits_{m=0}^{k} 
\left(\frac{m+2p}{2p}\right)^{p-1} \alpha_m.    
\end{equation*}
Therefore,
\begin{equation}
\label{eq_sumb}
    \sum\limits_{m=0}^{k} B_m \leq
\left(\frac{k+2p}{2p}\right)^{p-1} A_k.   
\end{equation}

\underline{Proof of \eqref{eq_fek}.}
Now using \eqref{eq_UIGM_A_B},\eqref{eq_suma} and \eqref{eq_sumb} we can get convergence rate.
\begin{equation*}
\begin{aligned}
    F(y_k)-F^{\ast} &\overset{\eqref{eq_UIGM_A_B}}{\leq} \frac{d(x^{\ast})}{A_k} + \frac{2\delta_u}{A_k}\sum \limits_{i=0}^k B_i+\frac{\varepsilon}{2}+\frac{(2k+1)\delta_{p}}{A_k}\\
    &\overset{\eqref{eq_suma}}{\leq} \frac{4^{\frac{2p\nu-\nu+1}{1+\nu}}M_\nu^{\frac{2}{1+\nu}}d(x^{\ast})}{\varepsilon^{\frac{1-\nu}{1+\nu}}(k+2)^{\frac{2p\nu-\nu+1}{1+\nu}}} + \frac{2\delta_u}{A_k}\sum \limits_{i=0}^k B_i+\frac{4^{\frac{2p\nu-\nu+1}{1+\nu}}M_\nu^{\frac{2}{1+\nu}}(2k+1)}{\varepsilon^{\frac{1-\nu}{1+\nu}}(k+2)^{\frac{2p\nu-\nu+1}{1+\nu}}} \delta_{p}+\frac{\varepsilon}{2}\\
    &\leq \frac{16 M_\nu^{\frac{2}{1+\nu}}d(x^{\ast})}{\varepsilon^{\frac{1-\nu}{1+\nu}}(k+2)^{\frac{2p\nu-\nu+1}{1+\nu}}} + \frac{2\delta_u}{A_k}\sum \limits_{i=0}^k B_i+\frac{32 M_\nu^{\frac{2}{1+\nu}}}{\varepsilon^{\frac{1-\nu}{1+\nu}}(k+2)^{\frac{2\nu(p-1)}{1+\nu}}} \delta_{p}+\frac{\varepsilon}{2}\\
     &\overset{\eqref{eq_sumb}}{\leq} \frac{16M_\nu^{\frac{2}{1+\nu}}d(x^{\ast})}{\varepsilon^{\frac{1-\nu}{1+\nu}}(k+2)^{\frac{2p\nu-\nu+1}{1+\nu}}} + 2\delta_u \left(\frac{k+2p}{2p}\right)^{p-1}+\frac{32M_\nu^{\frac{2}{1+\nu}}}{\varepsilon^{\frac{1-\nu}{1+\nu}}(k+2)^{\frac{2\nu(p-1)}{1+\nu}}} \delta_{p}+\frac{\varepsilon}{2}\\
    &\leq  \frac{16 M_\nu^{\frac{2}{1+\nu}}d(x^{\ast})}{\varepsilon^{\frac{1-\nu}{1+\nu}}(k+2)^{\frac{2p\nu-\nu+1}{1+\nu}}}   + 4k^{p-1}\delta_u +\frac{32 M_\nu^{\frac{2}{1+\nu}}}{\varepsilon^{\frac{1-\nu}{1+\nu}}(k+2)^{\frac{2\nu(p-1)}{1+\nu}}} \delta_{p}+\frac{\varepsilon}{2}.
\end{aligned}
\end{equation*}

Since UIGM does not use $\nu$ as a parameter, we get
\begin{equation*}
    F(y_k)-F^{\ast} \leq \inf\limits_{\nu\in[0,1]}\left(\frac{16 M_\nu^{\frac{2}{1+\nu}}d(x^{\ast})}{\varepsilon^{\frac{1-\nu}{1+\nu}}(k+2)^{\frac{2p\nu-\nu+1}{1+\nu}}} +\frac{32M_\nu^{\frac{2}{1+\nu}}}{\varepsilon^{\frac{1-\nu}{1+\nu}}(k+2)^{\frac{2\nu(p-1)}{1+\nu}}} \delta_{p}\right)+ 4k^{p-1}\delta_u   + \frac{\varepsilon}{2}.
\end{equation*}
\end{proof}

\begin{corollary}
 At each iteration $m \geq 0$ of UIGM with sequences $\{\alpha_m\}$, $\{B_m\}$, $m\geq 0$ chosen in accordance with \eqref{eq_apower} and \eqref{eq_bpower}, we have,
for any $p\in [1,2]$,
\begin{equation*}
\delta_{m}= O\left(\frac{\varepsilon}{m^{p-1}}\right)+\delta_{u}.
\end{equation*}

\end{corollary}
\begin{proof}
By \eqref{eq_dc}, we have
\begin{equation*}
\delta_{m}-\delta_u=\frac{\varepsilon\alpha_{m}}{4B_{m}}
=\frac{\varepsilon}{4\left(\frac{m+2p}{2p}\right)^{p-1}}=O\left(\frac{\varepsilon}{(m)^{p-1}}\right).
\end{equation*}

\end{proof}

From the rate of convergence \eqref{eq_fek} and the fact that UIGM does not include $\nu$ as a parameter, we get the following estimation for the number of iterations, which are necessary for getting first term of \eqref{eq_fek} smaller than $\varepsilon/ 6$ we need:
\begin{equation*}
 N=O \left[\inf\limits_{\nu\in[0,1]} \left(\frac{M_{\nu}^{2}d(x^{\ast})^{1+\nu}}{\varepsilon^{2}} \right)^{\frac{1}{2p\nu-\nu+1}}  \right].
\end{equation*}
The dependence of this bound in smoothness parameters is optimal (see \cite{nemirovsky1983problem}).

 Let's compare our method and convergence rate with existing optimal methods.  We assume that $N$ the number of iterations, then $
    F(y_N)-F^{\ast} \leq B(N) + C(N)\delta_p + D(N)\delta_u    + \frac{\varepsilon}{2}.$ Here $B(N)$ is an accuracy of our method, $C(N)$ is a speed of collecting prox-mapping error and $D(N)$ is a speed of collecting oracle error. As a result we get next table.
\begin{table}[htbp]
\begin{center}
\begin{tabular}{ | c | c | c | c |}
\hline
$(\nu, p)$ & $B(N)$ & $C(N)$ & $D(N)$ \\ \hline
$(0, 1)$ & 
$O\left(\frac{M_0d(x^{\ast})^{1/2} }{N^{1/2}} \right)$
& $O\left( M_0 d(x^{\ast})^{1/2} N^{1/2}  \right)$ 
& $O\left(1\right)$ \\ \hline

$(1, 1)$ & 
$O\left(\frac{M_1 d(x^{\ast}) }{N}\right)$
& $O\left( M_1 d(x^{\ast}) \right)$ 
& $O\left( 1 \right)$ \\ \hline

$(1, 2)$ & 
$O\left( \frac{M_1d(x^{\ast})  }{N^2}\right)$
& $O\left(\frac{M_1d(x^{\ast})  }{N}\right)$ 
& $O\left(N\right)$ \\ \hline

\end{tabular}

\end{center}
\end{table}

 For non-smooth functions $(\nu=0)$, the convergence rate of UIGM for any $p\in[1,2]$ agrees with rate of convergence of subgradient methods. This methods are robust for oracle error, but collect prox-mapping error.
For smooth functions $(\nu=1)$ and $p=1$ UIGM has  
the same convergence rate as a dual gradient method. This method is robust both for oracle error and prox-mapping error.
And for $p=2$ UIGM has the same rate as a fast gradient method. This method collects oracle error but kill prox-mapping error. This table shows three main regimes for UIGM and how it corresponds with classical methods.

\section{Accelerating UIGM for strongly convex functions}
In this section, we consider problem \eqref{eq_pr} with additional assumption of strong convexity of the objective $F$
\begin{equation}
    \label{strong}
F(y) \geq F(x) +\left\langle \nabla F(x),y-x \right\rangle+\frac{\mu}{2}\|y-x \|^2_E, \quad \forall x,y \in Q,
\end{equation}
where the constant $\mu > 0$ is assumed to be known.
We also assume that the chosen prox-function has quadratic growth
\begin{equation}
    \label{eq_wn}
    d (x)\leq \frac{\Omega}{2} \|x\|^2_E,
\end{equation}
where $\Omega$ is some dimensional-dependent constant, and that we are given a starting point $x_0$ and a number $R_0$ such that
\begin{equation}
    \label{eq_R0}
    \|x_0-x^{\ast}\|^2_E\leq R^{2}_{0},
\end{equation}
where $x^{\ast}$ is an optimal point in \eqref{eq_pr}.

\begin{algorithm}[H]
    \caption{Restart UIGM}
    \label{UFGMR}
    \begin{algorithmic}[1]
        \Require $\mu$ -- strong convexity parameter, $\Omega$ -- quadratic growth constant, $\varepsilon$ -- desired accuracy, $x_0$ -- starting point.
        \State Set $d_0(x) = d(x-x_0)$.
        \For{$m= 1, \dots$}    
            \While{$ 2\Omega > \mu A_k$}
                \State Run UIGM with accuracy $\varepsilon$ and prox-function $d_{m-1}(x)$.
            \EndWhile
            \State Set $x_m=y_{k} $.
            \State Set $d_m(x)=d\left(x - x_m\right)$.
        \EndFor
    \end{algorithmic}
\end{algorithm}

\begin{theorem}
Let $F$ be strongly convex with constant $\mu$ and \eqref{eq_wn}, \eqref{eq_R0} hold. Then, for any $m \geq 0$ restarts of UIGM with power policy \eqref{eq_apower}, \eqref{eq_bpower}, 
\begin{equation}
\label{eq_strongF}
\begin{aligned}
&F(x_m)-F(x^{\ast}) \leq \mu R^2_0 2^{-m-1}+\\ &+ 2\left( \frac{\varepsilon}{2} + 2\left\lceil \left( \frac{2^{\nu + 1} \Omega^{\nu + 1}  M_\nu^{2}}{\mu^{\nu + 1}  \varepsilon^{1 - \nu}}\right)\right\rceil^{\frac{p-1}{2p\nu-\nu+1}} \delta_u+\frac{p 2^{\frac{2p\nu+ 2}{2p\nu-\nu+1}} \mu^{\frac{2\nu(p-1)}{2p\nu-\nu+1}}  M_\nu^{\frac{2}{2p\nu-\nu+1}}}{\Omega^{\frac{2\nu(p-1)}{2p\nu-\nu+1}}  \varepsilon^{\frac{1-\nu}{2p\nu-\nu+1}}}\delta_{p}   \right),
\end{aligned}
\end{equation}
\begin{equation}
\label{eq_strongX}
\begin{aligned}
&\|x_m-x^{\ast} \|^2_E  \leq R_m^2=R^2_0 2^{-m}+\\ 
&+\frac{4}{\mu} \left( \frac{\varepsilon}{2}+ 2\left\lceil \left( \frac{2^{\nu + 1} \Omega^{\nu + 1}  M_\nu^{2}}{\mu^{\nu + 1}  \varepsilon^{1 - \nu}}\right)\right\rceil^{\frac{p-1}{2p\nu-\nu+1}} \delta_u+\frac{p 2^{\frac{2p\nu+ 2}{2p\nu-\nu+1}} \mu^{\frac{2\nu(p-1)}{2p\nu-\nu+1}}  M_\nu^{\frac{2}{2p\nu-\nu+1}}}{\Omega^{\frac{2\nu(p-1)}{2p\nu-\nu+1}}  \varepsilon^{\frac{1-\nu}{2p\nu-\nu+1}}}\delta_{p}  \right).
\end{aligned}
\end{equation}
 
\end{theorem}
\begin{proof}
From (\ref{strong}), we have
\[ \frac{\mu}{2}\|x_m-x^{\ast} \|^2_E \leq \left\langle \nabla F(x^{\ast}),x_m-x^{\ast} \right\rangle+\frac{\mu}{2}\|x_m-x^{\ast} \|^2_E \leq  F(x_m)-F(x^{\ast}). 
\]
Then, by the first-order optimality condition,
\[\frac{\mu}{2}\|x_m-x^{\ast} \|^2_E \leq F(x_m)-F(x^{\ast}). \]
From this fact and \eqref{eq_strongF} we can easily prove \eqref{eq_strongX}.

To prove \eqref{eq_strongF}, we prove a stronger inequality by induction
\begin{equation}
\label{eq_strongF2}
\begin{aligned}
&F(x_m)-F(x^{\ast}) \leq \mu R^2_0 2^{-m-1}+ 2\left( 1-2^{-m} \right) \left( \frac{\varepsilon}{2} +\right.\\ &\left. + 2\left\lceil \left( \frac{2^{\nu + 1} \Omega^{\nu + 1}  M_\nu^{2}}{\mu^{\nu + 1}  \varepsilon^{1 - \nu}}\right)\right\rceil^{\frac{p-1}{2p\nu-\nu+1}} \delta_u+\frac{p 2^{\frac{2p\nu+ 2}{2p\nu-\nu+1}} \mu^{\frac{2\nu(p-1)}{2p\nu-\nu+1}}  M_\nu^{\frac{2}{2p\nu-\nu+1}}}{\Omega^{\frac{2\nu(p-1)}{2p\nu-\nu+1}}  \varepsilon^{\frac{1-\nu}{2p\nu-\nu+1}}}\delta_{p}   \right).
\end{aligned}
\end{equation}

For $m=1$, we have 
\begin{equation*}
\begin{aligned}
F(x_1)-F(x^{\ast}) &\overset{\eqref{eq_UIGM_A_B}}{\leq}\frac{d_0 (x^{\ast})}{A_{k}}+ 2 \frac{ \sum \limits_{i=0}^k B_i }{A_k}\delta_u+\frac{(2k+1)\delta_{p}}{A_{k}}+\frac{\varepsilon}{2} \\
&\overset{\eqref{eq_wn}}{\leq} \frac{\Omega \|x_0-x^{\ast}\|^2_E}{2A_{k}}+ 2 \frac{ \sum \limits_{i=0}^k B_i }{A_k}\delta_u+\frac{(2k+1)\delta_{p}}{A_{k}}+\frac{\varepsilon}{2}\\
&\overset{\eqref{eq_R0}}{\leq} \frac{\Omega R_0^2}{2A_{k}}+ 2 \frac{ \sum \limits_{i=0}^k B_i }{A_k}\delta_u+\frac{(2k+1)\delta_{p}}{A_{k}}+\frac{\varepsilon}{2}\\ 
&\overset{\eqref{eq_sumb}}{\leq} \frac{\Omega R_0^2}{2A_{k}}+ 2 \left(\frac{k+2p}{2p}\right)^{p-1} \delta_u+\frac{(2k+1)\delta_{p}}{A_{k}}+\frac{\varepsilon}{2}.
\end{aligned}
\end{equation*}
By the condition on the Step 3 of the algorithm, we have 
\begin{equation}
\label{eq_rest_ineq}
A_{k-1} < \frac{2\Omega}{\mu }\leq A_k,
\end{equation}
and 
\begin{align}
\frac{2\Omega}{\mu}\geq  A_{k-1} \overset{\eqref{eq_suma}}{\geq} \frac{\varepsilon^{\frac{1-\nu}{1+\nu}}}{ M_\nu^{\frac{2}{1+\nu}}}\left(\frac{k-1+2p}{2p}\right)^{\frac{2p\nu-\nu+1}{1+\nu}},\\
\Rightarrow \quad \frac{2^{\frac{1+\nu}{2p\nu-\nu+1}} \Omega^{\frac{1+\nu}{2p\nu-\nu+1}}  M_\nu^{\frac{2}{2p\nu-\nu+1}}}{\mu^{\frac{1+\nu}{2p\nu-\nu+1}}  \varepsilon^{\frac{1-\nu}{2p\nu-\nu+1}}} \geq \left(\frac{k-1+2p}{2p}\right),\\
\Rightarrow \quad \frac{p 2^{\frac{2p\nu + 2}{2p\nu-\nu+1}} \Omega^{\frac{1+\nu}{2p\nu-\nu+1}}  M_\nu^{\frac{2}{2p\nu-\nu+1}}}{\mu^{\frac{1+\nu}{2p\nu-\nu+1}}  \varepsilon^{\frac{1-\nu}{2p\nu-\nu+1}}}+1-2p \geq k,\\
\Rightarrow \quad \frac{p 2^{\frac{2p\nu + 2}{2p\nu-\nu+1}} \Omega^{\frac{1+\nu}{2p\nu-\nu+1}}  M_\nu^{\frac{2}{2p\nu-\nu+1}}}{\mu^{\frac{1+\nu}{2p\nu-\nu+1}}  \varepsilon^{\frac{1-\nu}{2p\nu-\nu+1}}}+1-2p \geq k.\\
\end{align}
Hence,
\begin{align}
\left(\frac{k+2p}{2p}\right)^{p-1}\leq\left\lceil \left( \frac{2^{\nu + 1} \Omega^{\nu + 1}  M_\nu^{2}}{\mu^{\nu + 1}  \varepsilon^{1 - \nu}}\right)\right\rceil^{\frac{p-1}{2p\nu-\nu+1}} \label{eq_2pk},\\
\Rightarrow \quad 2k+1\leq \frac{p 2^{\frac{4p\nu -\nu + 3}{2p\nu-\nu+1}} \Omega^{\frac{1+\nu}{2p\nu-\nu+1}}  M_\nu^{\frac{2}{2p\nu-\nu+1}}}{\mu^{\frac{1+\nu}{2p\nu-\nu+1}}  \varepsilon^{\frac{1-\nu}{2p\nu-\nu+1}}}\label{eq_2k}.
\end{align}

Finally, we have
\begin{equation*}
\begin{aligned}
&F(x_1)-F(x^{\ast})\overset{\eqref{eq_2pk}, \eqref{eq_2k}}{\leq} \frac{\Omega R_0^2}{2A_{k}}+ 2\left\lceil \left( \frac{2^{\nu + 1} \Omega^{\nu + 1}  M_\nu^{2}}{\mu^{\nu + 1}  \varepsilon^{1 - \nu}}\right)\right\rceil^{\frac{p-1}{2p\nu-\nu+1}} \delta_u+\frac{p 2^{\frac{4p\nu -\nu + 3}{2p\nu-\nu+1}} \Omega^{\frac{1+\nu}{2p\nu-\nu+1}}  M_\nu^{\frac{2}{2p\nu-\nu+1}}}{\mu^{\frac{1+\nu}{2p\nu-\nu+1}}  \varepsilon^{\frac{1-\nu}{2p\nu-\nu+1}}A_k}\delta_{p} +\frac{\varepsilon}{2}\\
&\leq \mu R^2_0 2^{-2} + 2\left(1-2^{-1} \right)  \left( \frac{\varepsilon}{2}+ 2\left\lceil \left( \frac{2^{\nu + 1} \Omega^{\nu + 1}  M_\nu^{2}}{\mu^{\nu + 1}  \varepsilon^{1 - \nu}}\right)\right\rceil^{\frac{p-1}{2p\nu-\nu+1}} \delta_u+\frac{p 2^{\frac{2p\nu+ 2}{2p\nu-\nu+1}} \mu^{\frac{2\nu(p-1)}{2p\nu-\nu+1}}  M_\nu^{\frac{2}{2p\nu-\nu+1}}}{\Omega^{\frac{2\nu(p-1)}{2p\nu-\nu+1}}  \varepsilon^{\frac{1-\nu}{2p\nu-\nu+1}}}\delta_{p}  \right). \\
\end{aligned}
\end{equation*}
So \eqref{eq_strongF} is proved for $m=1$. Now we assume that \eqref{eq_strongF} holds for $m$ and prove that it holds for $m+1$.

From \eqref{eq_UIGM_A_B} we get
\begin{equation*}
\begin{aligned}
F(x_{m+1})-F(x^{\ast})&\overset{\eqref{eq_UIGM_A_B}}{\leq}\frac{d_m (x^{\ast})}{A_{k}}+ 2 \frac{ \sum \limits_{i=0}^k B_i }{A_k}\delta_u +\frac{(2k+1)\delta_{p}}{A_{k}} +\frac{\varepsilon}{2}\\
&\overset{\eqref{eq_wn}}{\leq} \frac{\Omega \|x_m-x^{\ast}\|^2_E}{2A_{k}}+ 2 \frac{ \sum \limits_{i=0}^k B_i }{A_k}\delta_u+\frac{(2k+1)\delta_{p}}{A_{k}}+\frac{\varepsilon}{2}\\
&\overset{\eqref{eq_strongX}}{\leq} \frac{\Omega R_m^2}{2A_{k}}+ 2 \frac{ \sum \limits_{i=0}^k B_i }{A_k}\delta_u+\frac{(2k+1)\delta_{p}}{A_{k}}+\frac{\varepsilon}{2} \\
&\overset{\eqref{eq_sumb}}{\leq} \frac{\Omega R_m^2}{2A_{k}}+ 2 \left(\frac{k+2p}{2p}\right)^{p-1} \delta_u+\frac{(2k+1)\delta_{p}}{A_{k}}+\frac{\varepsilon}{2}\\
&\overset{\eqref{eq_2pk},\eqref{eq_2k}}{\leq} \frac{\Omega R_m^2}{2A_{k}} + 2\left\lceil \left( \frac{2^{\nu + 1} \Omega^{\nu + 1}  M_\nu^{2}}{\mu^{\nu + 1}  \varepsilon^{1 - \nu}}\right)\right\rceil^{\frac{p-1}{2p\nu-\nu+1}} \delta_u \\
&+\frac{p 2^{\frac{4p\nu -\nu + 3}{2p\nu-\nu+1}} \Omega^{\frac{1+\nu}{2p\nu-\nu+1}}  M_\nu^{\frac{2}{2p\nu-\nu+1}}}{\mu^{\frac{1+\nu}{2p\nu-\nu+1}}  \varepsilon^{\frac{1-\nu}{2p\nu-\nu+1}}A_k}\delta_{p} +\frac{\varepsilon}{2} \\
&\overset{\eqref{eq_rest_ineq}}{\leq} \frac{\mu R_m^2}{4}  + 2\left\lceil \left( \frac{2^{\nu + 1} \Omega^{\nu + 1}  M_\nu^{2}}{\mu^{\nu + 1}  \varepsilon^{1 - \nu}}\right)\right\rceil^{\frac{p-1}{2p\nu-\nu+1}} \delta_u \\
&+\frac{p 2^{\frac{2p\nu+ 2}{2p\nu-\nu+1}} \mu^{\frac{2\nu(p-1)}{2p\nu-\nu+1}}  M_\nu^{\frac{2}{2p\nu-\nu+1}}}{\Omega^{\frac{2\nu(p-1)}{2p\nu-\nu+1}}  \varepsilon^{\frac{1-\nu}{2p\nu-\nu+1}}}\delta_{p} +\frac{\varepsilon}{2} \\
\end{aligned}
\end{equation*}
\begin{equation*}
\begin{aligned}
&\overset{\eqref{eq_strongX}}{\leq} \mu R^2_0 2^{-m-2}+\frac{4 \mu}{4\mu}\left( 1-2^{-m} \right) \left( \frac{\varepsilon}{2}+ 2\left\lceil \left( \frac{2^{\nu + 1} \Omega^{\nu + 1}  M_\nu^{2}}{\mu^{\nu + 1}  \varepsilon^{1 - \nu}}\right)\right\rceil^{\frac{p-1}{2p\nu-\nu+1}} \delta_u \right. \\ 
&+\left.  \frac{p 2^{\frac{2p\nu+ 2}{2p\nu-\nu+1}} \mu^{\frac{2\nu(p-1)}{2p\nu-\nu+1}}  M_\nu^{\frac{2}{2p\nu-\nu+1}}}{\Omega^{\frac{2\nu(p-1)}{2p\nu-\nu+1}}  \varepsilon^{\frac{1-\nu}{2p\nu-\nu+1}}}\delta_{p}\right) + \frac{\varepsilon}{2}+ 2\left\lceil \left( \frac{2^{\nu + 1} \Omega^{\nu + 1}  M_\nu^{2}}{\mu^{\nu + 1}  \varepsilon^{1 - \nu}}\right)\right\rceil^{\frac{p-1}{2p\nu-\nu+1}} \delta_u  \\ 
&+  \frac{p 2^{\frac{2p\nu+ 2}{2p\nu-\nu+1}} \mu^{\frac{2\nu(p-1)}{2p\nu-\nu+1}}  M_\nu^{\frac{2}{2p\nu-\nu+1}}}{\Omega^{\frac{2\nu(p-1)}{2p\nu-\nu+1}}  \varepsilon^{\frac{1-\nu}{2p\nu-\nu+1}}}\delta_{p} \\
&\leq \mu R^2_0 2^{-m-2}+\left( 2-2^{-m} \right) \left( \frac{\varepsilon}{2}+ 2\left\lceil \left( \frac{2^{\nu + 1} \Omega^{\nu + 1}  M_\nu^{2}}{\mu^{\nu + 1}  \varepsilon^{1 - \nu}}\right)\right\rceil^{\frac{p-1}{2p\nu-\nu+1}} \delta_u \right. \\ 
&+\left.  \frac{p 2^{\frac{2p\nu+ 2}{2p\nu-\nu+1}} \mu^{\frac{2\nu(p-1)}{2p\nu-\nu+1}}  M_\nu^{\frac{2}{2p\nu-\nu+1}}}{\Omega^{\frac{2\nu(p-1)}{2p\nu-\nu+1}}  \varepsilon^{\frac{1-\nu}{2p\nu-\nu+1}}}\delta_{p}\right)\\
&\leq \mu R^2_0 2^{-m-2}+ 2\left(1-2^{-(m+1)} \right)  \left( \frac{\varepsilon}{2}+ 2\left\lceil \left( \frac{2^{\nu + 1} \Omega^{\nu + 1}  M_\nu^{2}}{\mu^{\nu + 1}  \varepsilon^{1 - \nu}}\right)\right\rceil^{\frac{p-1}{2p\nu-\nu+1}} \delta_u \right. \\ 
&+\left.  \frac{p 2^{\frac{2p\nu+ 2}{2p\nu-\nu+1}} \mu^{\frac{2\nu(p-1)}{2p\nu-\nu+1}}  M_\nu^{\frac{2}{2p\nu-\nu+1}}}{\Omega^{\frac{2\nu(p-1)}{2p\nu-\nu+1}}  \varepsilon^{\frac{1-\nu}{2p\nu-\nu+1}}}\delta_{p}\right).
\end{aligned}
\end{equation*}
So we have obtained that \eqref{eq_strongF} holds for m+1 and by induction it holds for all $m \geq 1$

\end{proof}

\begin{corollary}
For getting $(\varepsilon+C_u)$-solution of problem \eqref{eq_pr}, where
\begin{equation*}
C_u=2\left(\frac{\varepsilon}{2} + 2\left\lceil \left( \frac{2^{\nu + 1} \Omega^{\nu + 1}  M_\nu^{2}}{\mu^{\nu + 1}  \varepsilon^{1 - \nu}}\right)\right\rceil^{\frac{p-1}{2p\nu-\nu+1}} \delta_u +  \frac{p 2^{\frac{2p\nu+ 2}{2p\nu-\nu+1}} \mu^{\frac{2\nu(p-1)}{2p\nu-\nu+1}}  M_\nu^{\frac{2}{2p\nu-\nu+1}}}{\Omega^{\frac{2\nu(p-1)}{2p\nu-\nu+1}}  \varepsilon^{\frac{1-\nu}{2p\nu-\nu+1}}}\delta_{p} \right),
\end{equation*}
we need
\begin{equation} 
\label{eq_strongN}
\tilde{l}=\left\lceil \log \left( \frac{\mu R_0^2}{2\varepsilon} \right) \right\rceil
\end{equation}
restarts and
\begin{equation}
\tilde{k}\leq\inf \limits_{0\leq\nu\leq 1}\left(\frac{ \Omega^{1+\nu} 2^{4p\nu-\nu+3} M_\nu^{2}}{\mu^{1+\nu}\varepsilon^{1-\nu}}\right)^{\frac{1}{2p\nu-\nu+1}}+1
\end{equation}
iterations of UIGM per iteration.
The total number of UIGM iterations is not more than
\[N=\left(\inf \limits_{0\leq\nu\leq 1}\left(\frac{\Omega^{1+\nu} 2^{4p\nu-\nu+3} M_\nu^{2}}{\mu^{1+\nu}\varepsilon^{1-\nu}}\right)^{\frac{1}{2p\nu-\nu+1}}+1\right)\cdot \left\lceil \log \left( \frac{\mu R_0^2}{2\varepsilon} \right) \right\rceil.
\]
\end{corollary}

\begin{proof}

\begin{equation*}
\begin{aligned}
&F(x_{\tilde{l}})-F(x^{\ast}) \overset{\eqref{eq_strongF}}{\leq} \mu R^2_0 2^{-\tilde{l}-1}+C_u\overset{\eqref{eq_strongN}}{\leq} 
\varepsilon+C_u.
\end{aligned}
\end{equation*}

We now estimate the total number of UIGM iterations, which is sufficient to obtain $(\varepsilon+C_u)$-solution.
First, we estimate the number $\tilde{k}$ of UIGM iterations at each restart. By the stopping condition for the restart, we have
\begin{equation*}
\begin{aligned}
A_{\tilde{k}} \geq \frac{2 \Omega}{\mu} \overset{\eqref{eq_suma}}{>} A_{\tilde{k}-1 }\geq \frac{\varepsilon^{\frac{1-\nu}{1+\nu}}}{ M_\nu^{\frac{2}{1+\nu}}}\left(\frac{\tilde{k}-1}{4}\right)^{\frac{2p\nu-\nu+1}{1+\nu}},\\
\Rightarrow \quad \frac{2 \Omega \, 2^{\frac{4p\nu-2\nu+2}{1+\nu}} M_\nu^{\frac{2}{1+\nu}}}{\mu\,\varepsilon^{\frac{1-\nu}{1+\nu}}}\geq \left(\tilde{k}-1\right)^{\frac{2p\nu-\nu+1}{1+\nu}},\\
\Rightarrow \quad \left(\frac{\Omega^{1+\nu} 2^{4p\nu-\nu+3} M_\nu^{2}}{\mu^{1+\nu}\varepsilon^{1-\nu}}\right)^{\frac{1}{2p\nu-\nu+1}}\geq \tilde{k}-1.    
\end{aligned}
\end{equation*}
Since the algorithm does not use any particular choice of $\nu$, we have
\begin{equation*}
\tilde{k}\leq\inf \limits_{0\leq\nu\leq 1}\left(\frac{\Omega^{1+\nu} 2^{4p\nu-\nu+3} M_\nu^{2}}{\mu^{1+\nu}\varepsilon^{1-\nu}}\right)^{\frac{1}{2p\nu-\nu+1}}+1.
\end{equation*}

Then the total number of UIGM is not more than $N=\tilde{k} \cdot \tilde{l}$, and we have
\[N=\left(\inf \limits_{0\leq\nu\leq 1}\left(\frac{\Omega^{1+\nu} 2^{4p\nu-\nu+3} M_\nu^{2}}{\mu^{1+\nu}\varepsilon^{1-\nu}}\right)^{\frac{1}{2p\nu-\nu+1}}+1\right)\cdot \left\lceil \log \left( \frac{\mu R_0^2}{2\varepsilon} \right) \right\rceil.
\]
\end{proof}

Now we compare our result with existing methods in the same manner as for convex functions.  $
    С_u=  C\delta_p + D\delta_u    + \frac{\varepsilon}{2}.$ Here $C$ is a collecting prox-mapping error for given $\varepsilon,\mu$ and $D$ is a collecting oracle error for desired $\varepsilon,\mu$. As a result we get next table.

\begin{center}
\begin{tabular}{ | c | c | c | c |}
\hline
$(\nu, p)$ & $N$ & $C$ & $D$ \\ \hline
$(0, 1)$ & 
$O\left(\frac{\Omega M_0^{2}}{\mu\varepsilon}\cdot \log \left( \frac{\mu R_0^2}{2\varepsilon} \right) \right)$
& $O\left(\frac{M_0^2}{ \varepsilon}  \right)$
& $O   \left( 1\right)  $
 \\ \hline

$(1, 1)$ & 
$O\left(\frac{\Omega M_1}{\mu}\cdot \log \left( \frac{\mu R_0^2}{2\varepsilon}  \right) \right)$
& $O\left(M_1\right)$
& $O   \left( 1\right)  $ \\ \hline

$(1, 2)$ & 
$O\left(\left(\frac{\Omega M_1}{\mu}\right)^{\frac{1}{2}}\cdot \log \left( \frac{\mu R_0^2}{2\varepsilon}  \right) \right)$
& $O\left(\frac{ \mu  M_1}{\Omega}  \right)^{\frac{1}{2}}$
& $O   \left( \frac{ \Omega  M_1}{\mu }\right)^{\frac{1}{2}}  $ \\ \hline

\end{tabular}
\end{center}

 For non-smooth functions $(\nu=0)$, the convergence rate of Restart UIGM for any $p\in[1,2]$ agrees with rate of convergence of subgradient methods. This methods are robust for oracle error, but collect prox-mapping error.
For smooth functions $(\nu=1)$ and $p=1$ Restart UIGM has  
the same convergence rate as a dual gradient method. This method is robust both for oracle error and prox-mapping error.
And for $p=2$ Restart UIGM has the same rate as a fast gradient method. This method collects oracle error but kill prox-mapping error. This table shows three main regimes for Restart UIGM and how it corresponds with classical methods.

\section{Switching policy}
In this section we describe another variant of coefficient policy. The key observation is that Fast gradient method(FGM) accumulates the error, but converges faster and Dual gradient method(DGM) doesn't  accumulate the error, but works slower. That's why at the begging we make some steps of FGM until the error reaches some limit and then make only DGM steps. This policy was introduced in \cite{devolder2013intermediate}. Now we should understand, what is the limit. If we want to get the total error equal to $\varepsilon$, then the the error from inexactness should be $\varepsilon/2$. Now we describe this idea in more details.

Let the switching policy is
\begin{equation}
\label{switch_a}
    \alpha_k= \begin{cases} \frac{k+4}{4}\cdot\frac{1}{L_k} &k = 0, \ldots s \quad\text{--- FGM steps}\\
    c_k\cdot\frac{1}{L_k} &k = s+1, \ldots N  \quad \text{--- DGM steps}
    \end{cases},
\end{equation}
where $s$ is the moment of switching and $c_k$ is some constant, we will describe later how to choose them.

Firstly, we should prove, that the switching policy can be used in UIGM. Let check the correctness of inequalities \eqref{eq_ABineq} for the switching policy. For FGM steps it is easily follow from \eqref{power_ok}, because it is the power policy for $p=2$. For DGM steps we need to prove that
\begin{align*}
    0 < c_k \cdot \frac{1}{L_k} \leq c_k^2  \cdot \frac{1}{L_k} \leq A_{k-1} + c_k  \cdot \frac{1}{L_k}
\end{align*}
First two inequalities are satisfied if $c_k\geq 1$. So we get the first condition for $c_k$. The second condition comes from the last inequality because we need to get $c_k$ such that $c_k^2-c_k-A_{k-1}L_k \geq 0$. Hence
\begin{equation}
\label{switch_c1}
   1 \leq c_k \leq \frac{1+\sqrt{1+A_{k-1}L_k}}{2}
\end{equation}
So if these two conditions for $c_k$ are satisfied we prove that the switching policy can be used in UIGM.

Secondly, we should prove convergence of the switching policy. For that we need to satisfy two inequalities from \eqref{eq_UIGM_A_B}
\begin{align}
\label{switch_inex_err}
\frac{2\delta_u}{A_K}\sum\limits_{j=0}^{k} B_j \leq \frac{\varepsilon}{6}\\
\label{switch_prox_err}
    \frac{(2k+1)\delta_p}{A_k}\leq \frac{\varepsilon}{6}.
\end{align}
Note that from this two inequalities we get three main regimes:
\begin{itemize}
    \item Only FGM steps. In this case, $\delta_u << \varepsilon$ and $\delta_p << \varepsilon$, and both inequalities \eqref{switch_inex_err} and \eqref{switch_prox_err} never fail.
    \item Only DGM steps. In this case, $\delta_u$ is rather big and \eqref{switch_inex_err} fails on the first step, so we try our best and do only slow DGM steps.
    \item Switching on the moment $s$. In this case, we do some FGM steps until the moment $s$, when \eqref{switch_inex_err} fails at the first time and next do only DGM steps.
\end{itemize}
First two regimes are easy for understanding but for the last one we write more details.
Note that for FGM steps \eqref{switch_prox_err} always true because the left side decreases on each step. Note that from some moment $\sum\limits_{j=0}^k B_j/\sum\limits_{j=0}^k \alpha_j$ starts to increase and at the moment $s$ it reaches the limit $\frac{\varepsilon}{12 \delta_u}$. 
Now we need to check, that for DGM steps \eqref{switch_inex_err} will not fail.
\begin{align*}
    \sum\limits_{j=0}^{k} B_j&\leq \frac{\varepsilon}{12 \delta_u}\sum\limits_{j=0}^{k} \alpha_j\\
    \sum\limits_{j=0}^{k-1} B_j +B_{k}&\leq \frac{\varepsilon}{12 \delta_u}\left(\sum\limits_{j=0}^{k-1} \alpha_j + \alpha_{k}\right)
\end{align*}
We assume that on previous step \eqref{switch_inex_err} was correct, that's why we need
\begin{align*}
    B_{k}\leq \frac{\varepsilon}{12 \delta_u} \alpha_{k}\\
    \frac{c_{k}^2}{L_k}\overset{\eqref{switch_a}}{\leq} \frac{\varepsilon}{12 \delta_u} \frac{c_{k}}{L_k}\\
\end{align*}
So we get the third condition on $c_k$
\begin{equation}
\label{switch_c2}
    c_k\leq \frac{\varepsilon}{12 \delta_u}
\end{equation}
Hence when when we merge all conditions \eqref{switch_c1} and \eqref{switch_c2}, we get
\begin{equation}
\label{ck}
    c_k = \min \left(\frac{\varepsilon}{12 \delta_u}, \frac{1+\sqrt{1+A_{k-1}L_k}}{2} \right)
\end{equation}

As a result, we've proved that the switching policy can be used with UIGM. We've proved that UIGM converges, when we do FGM steps until at the moment $s$ \eqref{switch_inex_err} fails, then switch DGM with $c_k$ defined by \eqref{ck}. Note, that now our method needs to know only $\varepsilon, \delta_u, \delta_p$ and doesn't need $p$ as in the power policy, hence it converges as well or better than any $p$ for power policy.
\begin{theorem}
Assume that $f$ is a convex function with inexact first-order oracle. Then, for the sequence \eqref{switch_a}, the moment $s$ is first time when \eqref{switch_inex_err} fails and $c_k$ defined by \eqref{ck}, for all $k \geq 0$, 
\begin{equation}
\label{switch_fek}
    F(y_k)-F^{\ast} \leq \inf\limits_{p\in[1,2]}O\left[\inf\limits_{\nu\in[0,1]} \left(\frac{M_\nu^{\frac{2}{1+\nu}}d(x^{\ast})}{\varepsilon^{\frac{1-\nu}{1+\nu}}k^{\frac{2p\nu-\nu+1}{1+\nu}}} +\frac{M_\nu^{\frac{2}{1+\nu}}}{\varepsilon^{\frac{1-\nu}{1+\nu}}k^{\frac{2\nu(p-1)}{1+\nu}}} \delta_{p}\right)+ k^{p-1}\delta_u  \right] + \frac{\varepsilon}{2}.
\end{equation}
\end{theorem}

The same argumentation is correct also for strongly convex functions. So now we get fully adaptive and universal coefficient policy and method.

\section{Numerical illustration}
For  numerical illustration we choose a Poisson likelihood problems and as an application Positron Emission Tomography(PET). It plays an important role in medicine for detecting cancer and metabolic changes in human organ. PET can be treated as a Poisson likelihood model \cite{ben2001ordered},\cite{he2016fast}. The estimation of radioactivity density within an organ corresponds to the following convex  non-smooth optimization problem:
\begin{equation*}
\min\limits_{x\in \Delta_n} \sum_{i=1}^{m} \left[ \left[Ax \right]_i -w_i \log(\left[Ax \right]_i)  \right]
\end{equation*}
where $\Delta_n$ is a standard simplex. 
$A$  is a data and refers to the likelihood matrix known from geometry of detector, and $w$  is a data and refers to the vector of detected events, such that $w_i=[Ax]_i+b_i$, where $b_i$ is Poisson noise for any $1\leq i \leq m$. So we get a regression and our goal is to find $x$ from data.  For simplicity, we will not consider any penalty term for this application.Note that actually this problem has unbounded $M_\nu$ because $\nabla \log y = 1/y$ is unbounded in $y=0$. So here we assume that all points of our method are separated from zero and then $M_\nu$ is bounded by some constant.

 We assume, that tomographic scanner can have some small random and systematic errors, so we get inexact data and hence inexact function and gradient. So we get inexact oracle. If method converges with inexact data it means that we have robust system and even with errors we will get rather precise tomography.
 
 In this case, the entropy function $d(x)= \sum_{i=1}^n x_i \log(x_i)$ is a good choice for the simplex domain. Moreover, the prox-mapping can be computed by direct formula \cite{devolder2013intermediate}, which means that we have $\delta_p=$. If we choose another $d(x)$ it may be worse, because for the finding of the prox-mapping we need to solve additional optimization subproblem. For example we can approximately solve it by FGM with $\delta_p>0$, because this subproblem is strongly convex and that's why FGM converges fast.
 
 Code is written in Python 3. We conduct experiments using Ubuntu 14, machine: Intel Core i7-4510U CPU 2.00GHz 2.60GHz, 8Gb RAM. Matrix $A \in \mathbb{R}^{100 \times 200}$  and $w \in \mathbb{R}^{100}$ are generated unifomly randomly. For simplicity, we calculate inexact oracle as exact oracle plus the noise $\delta_u$. Desired accuracy is $\varepsilon=0.0001$

For small inexactness $\delta=0.001 \varepsilon$ UIGM give us next graphic.
\begin{figure}[H]
\label{im_delta0}
\begin{center}
\includegraphics[width=0.7\linewidth]{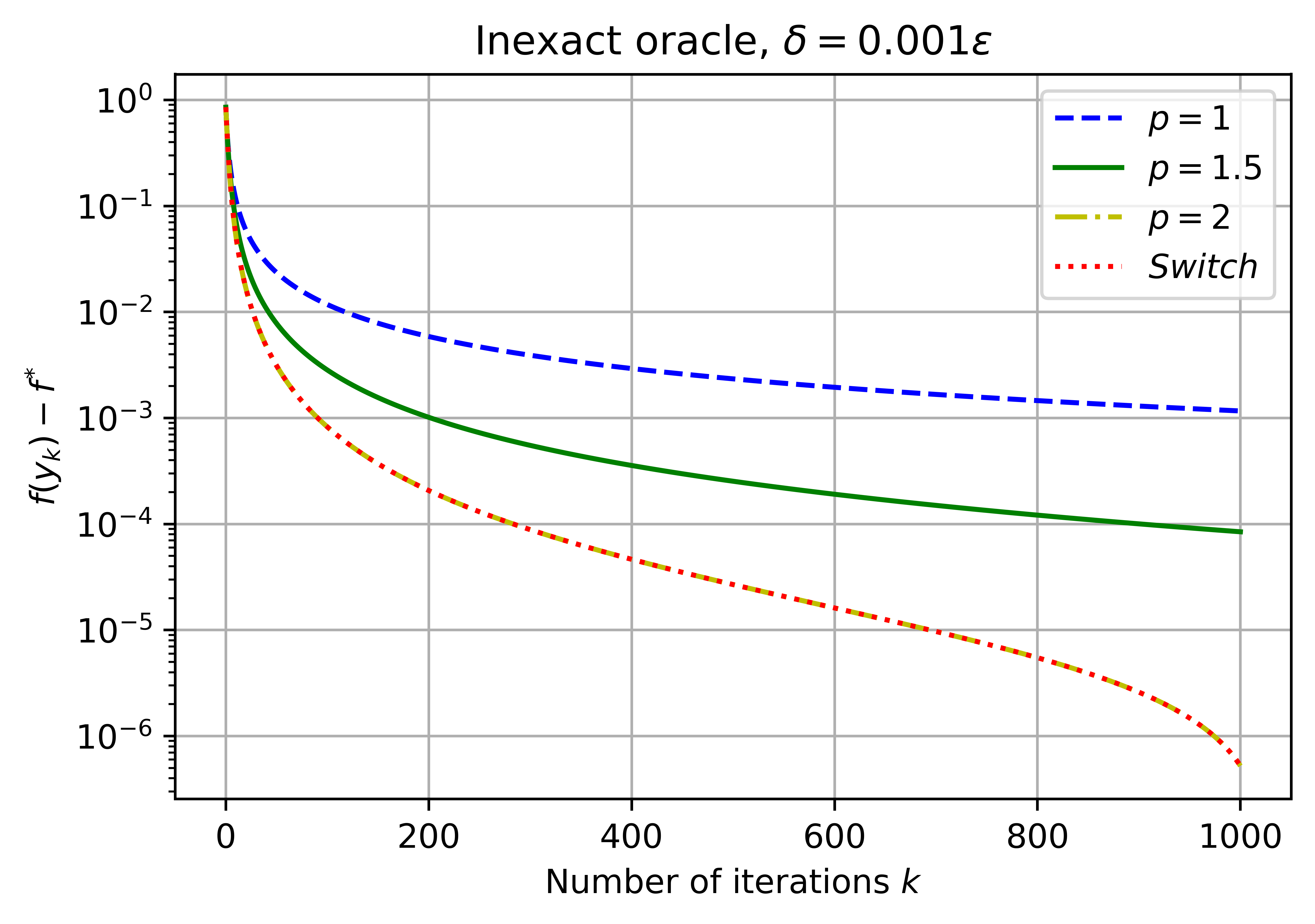}
\end{center}
\caption{Comparison of different power policies and switch policy for small inexactness}
\end{figure}
From this graphic we can see that, power policy with $p=2$ and switching policy are the fastest. For small inexactness all variants don't collect any noticeable error.

In next graphic, we can see, that for medium error $\delta=\varepsilon$ switching policy starts to work as power policy with $p=1$, because all our estimates of error collection come from theory but the real error in specific point can be less than theoretical estimate. Unfortunately we can't measure the real error.
\begin{figure}[H]
\label{im_eps10}
\begin{center}
\includegraphics[width=0.7\linewidth]{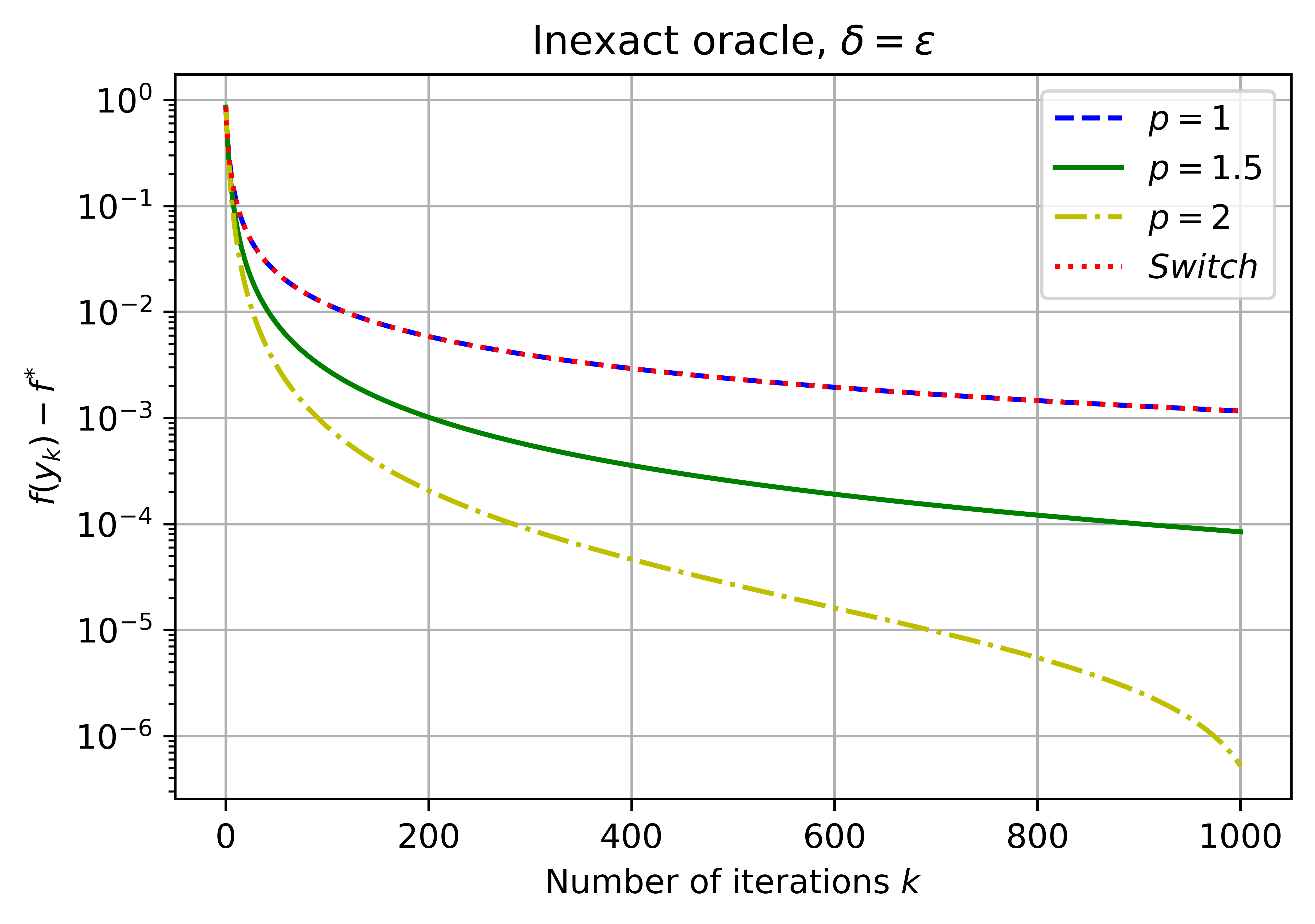}
\end{center}
\caption{Comparison of different power policies and switch policy for medium inexactness}
\end{figure}

For big error $\delta=1000\varepsilon$ the power policy with $p= 2$ collects error and works worse than the power policy $p=1.5$. So the method with intermediate rate is the best one, because it is rather fast and also robust. 
\begin{figure}[H]
\label{im_eps100}
\begin{center}
\includegraphics[width=0.7\linewidth]{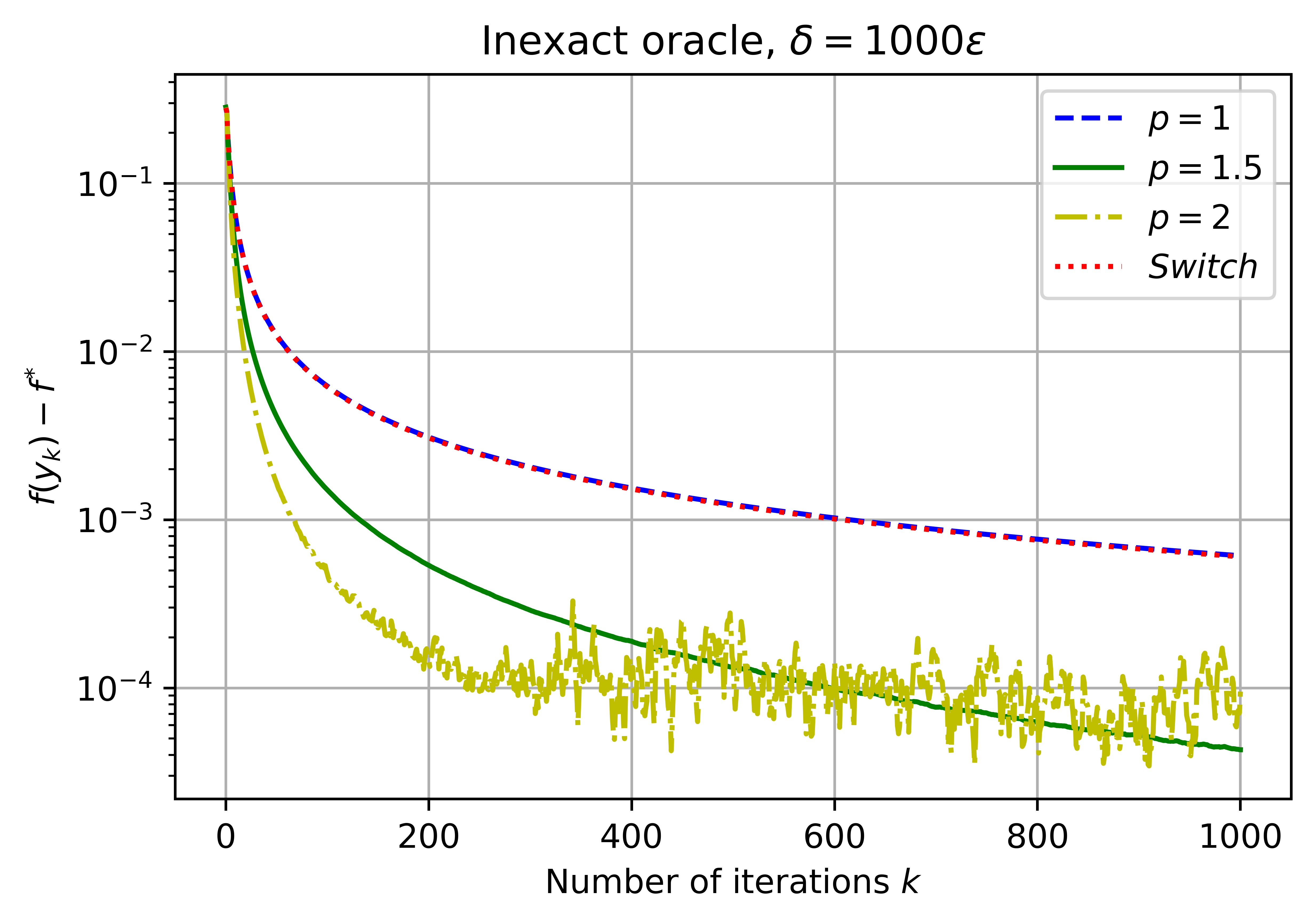}
\end{center}
\caption{Comparison of different power policies and switch policy for big inexactness}
\end{figure}
As a result, we get that UIGM for some intermediate $p$ can be better, than classical methods. Also we get that our method get a speed up for non-smooth problem in comparison with optimal DGM ($p=1$). Unfortunately in practice switching policy may be worse, than power policy because of uncertainty of real error.   

\section{Conclusion}

In this paper, we present new Universal Intermediate Gradient Method for convex optimization problem with inexact H\"older-continuous subgradient.  Our method enjoys both the universality with respect to smoothness of the problem and interpolates between Universal Gradient Method and Universal Fast Gradient Method, thus, allowing to balance the rate of convergence of the method and rate of the error accumulation. Under additional assumption of strong convexity of the objective, we show how the restart technique can be used to obtain an algorithm with faster rate of convergence.

We note that Theorem \ref{ufgm0} is primal-dual friendly. This means that, if UIGM is used to solve a problem, which is dual to a problem with linear constraints, it generates also a sequence of primal iterates and the rate for the primal-dual gap and linear constraints infeasibility is the same. This can be proved in the same way as in Theorem 2 of \cite{dvurechensky2017adaptive}. 
Also, based on the ideas from \cite{fercoq2016restarting,fercoq2017adaptive,roulet2017sharpness}, UIGM for the strongly convex case can be modified to work without exact knowledge of strong convexity parameter $\mu$.
Finally, similarly to \cite{gasnikov2015gradient,gasnikov2016universal,dvurechensky2016stochastic}, UIGM can be modified to solve convex problems with stochastic inexact oracle.

\textbf{Acknowledgements.}
This research was funded by Russian Science Foundation (project 17-11-01027).

\bibliographystyle{plain} 
\bibliography{biblioUIGM}

\end{document}